\documentclass[12pt,a4paper,twoside]{article}
\RequirePackage{filecontents}
\usepackage[utf8]{inputenc}
\usepackage[UKenglish]{babel}
\usepackage[T1]{fontenc}
\usepackage{fancyhdr}
\usepackage[left=2.5cm,right=2.5cm,top=1.5cm,bottom=1.5cm]{geometry}
\usepackage{etex}
\usepackage[etex=true,export]{adjustbox}
\usepackage{float}
\usepackage{xkeyval}
\usepackage{mathtools}
\usepackage{stmaryrd}
\usepackage{amsmath,accents}
\usepackage{amsthm}
\usepackage{amsfonts}
\usepackage{amssymb}
\usepackage{authblk}
\usepackage{amssymb, amsfonts, amsmath, amsthm, times}
\usepackage{epsfig}
\usepackage{lipsum}
\usepackage{graphicx}
\usepackage[sort, numbers]{natbib}
\usepackage{placeins}
\usepackage{hyperref}
\usepackage[dvipsnames]{xcolor}
\usepackage{ulem}
\usepackage{bbm}

\newtheoremstyle{theoremdd}
  {\topsep}
  {\topsep}
  {\normalfont}
  {0pt}
  {\normalfont}
  {.}
  {4pt}
  {\thmname{#1}\thmnumber{ #2}\textnormal{\thmnote{ (#3)}}}

\theoremstyle{theoremdd}
\newtheorem{mydef}{Definition}[section]
\newtheorem{rmk}{Remark}[section]
\newtheorem{empl}{Example}[section]

\newcommand{\R}{\mathbbm{R}}
\newcommand{\Z}{\mathbbm{Z}}

\theoremstyle{definition}
\newtheorem{conj}{Conjecture}[section]
\theoremstyle{plain}
\newtheorem{theo}{Theorem}[section]
\newtheorem{lem}[theo]{Lemma}
\newtheorem{prop}[theo]{Proposition}

\newtheorem{qst}{Question}

\binoppenalty=\maxdimen
\relpenalty=\maxdimen
\title{INFINITE FAMILIES OF NON-LEFT-ORDERABLE $L$-SPACES}
\author{HAMID ABCHIR and MOHAMMED SABAK}
\date{}
\begin{document}
\maketitle  
\begin{abstract}
For each connected alternating tangle, we provide an infinite family of non-left-orderable $L$-spaces. This gives further support for Conjecture \cite{boyer} of Boyer, Gordon, and Watson that is a rational homology $3$-sphere is an $L$-space if and only if it is non-left-orderable. These $3$-manifolds are obtained as Dehn fillings of the double branched covering of any alternating encircled tangle. We give a presentation of these non-left-orderable $L$-spaces as double branched coverings of $S^3$, branched over some specified links that turn out to be hyperbolic. We show that the obtained families include many non-Seifert fibered spaces. We also show that these families include many Seifert fibered spaces and give a surgery description for some of them. In the process we give another way to prove that the torus knots $T(2,2m+1)$ are $L$-space-knots as has already been shown by Ozsv{\'a}th and Szab{\'o} in \cite{ozsvath}.
\end{abstract}

\section{Introduction}
A group $G$ is said to be \textit{left-orderable} if there exists a total order $<$ on the elements of $G$ such that given any two elements $a$ and $b$ in $G$, if $a < b$ then $ca < cb$ for any $c \in G$. By convention, the trivial group is non-left-orderable.

One interesting problem studied by topologists is the relationship between the topology or geometry of a $3$-manifold and the left-orderability of its fundamental group. In 2005, Boyer, Rolfsen, and Wiest showed in \cite{boyer2} that if the fundamental group of a $3$-manifold $M$ is non-left-orderable, then $M$ is a rational homology $3$-sphere. 

An interesting familiy of rational homology $3$-spheres is that of $L$-spaces which was introduced in 2005 by Ozsv{\'a}th and Szab{\'o} \cite{ozsvath2}. Recall that a rational homology $3$-sphere $M$ is an $L$\textit{-space} if the rank of the Heegaard Floer homology group $\widehat{HF}(M)$ is equal to $\left| H_1(M;\mathbb{Z}) \right|$, the cardinal of the first homology group of $M$. Ozsv{\'a}th and Szab{\'o} showed in \cite{ozsvath} that Lens spaces are $L$-spaces. In particular, the $3$-sphere $S^3$ is an $L$-space. According to the following conjecture, it seems that $L$-spaces are the only rational homology $3$-spheres which satisfy the converse of the result showed by Boyer, Rolfsen, and Wiest cited above.

\begin{conj}[$L$-space conjecture \cite{boyer}]
The fundamental group of a rational homology $3$-sphere $M$ is non-left-orderable if and only if $M$ is an $L$-space.
\label{conj1}
\end{conj}
In 2013,  Boyer, Gordon, and Watson showed that this conjecture is true for Seifert fibered spaces and non-hyperbolic geometric 3–manifolds \cite{boyer}. Many known families of $L$-spaces have non-left-orderable fundamental groups. These families include the double branched coverings of non-split alternating links and those of genus two positive knots (\cite{greene}, \cite{ito}). On the other hand, there are many examples of $3$-manifolds with non-left-orderable fundamental groups detected by Dabkowski, Przytycki and Togha in \cite{dkabkowski}, Roberts and Shareshian in \cite{roberts}, and Roberts, Shareshian and Stein in \cite{roberts2}. Later on, Clay and Watson in \cite{clay}, and Peters in \cite{peters} showed that all these $3$-manifolds are $L$-spaces.

In \cite{ito}, Ito developed a method to show that the fundamental group of the double branched covering of a non-split link is non-left-ordrable by using the notion of \textit{Brunner's coarse presentation} that looks like usual group presentations. A Brunner's coarse presentation is given by a set of generators and relations, but inequalities are allowed as relations. It is derived from \textit{Brunner's presentation} introduced in \cite{brunner}.

In the present paper, we consider the double branched covering of an encircled alternating tangle whose boundary is a torus. Then by using some specific Dehn fillings we get rational homology $3$-spheres which will be $L$-spaces. We show that the fundamental groups of these $L$-spaces are non-left-orderable by using the coarse Brunner's presentation. So, we give further support for the $L$-space conjecture. Some of these obtained $3$-manifolds are non-Seifert fibered spaces.

More precisely, we consider the alternating encirclement of a connected alternating tangle $T$ denoted by $(B,\tau(T))$, where $B$ is the $3$-ball and $\tau(T)$ is the tangle $T$ encircled by a trivial simple close curved as in Fig. \ref{ae}. We denote by $\Sigma_2(B,\tau(T))$ the double branched covering of $(B,\tau(T))$. It is a $3$-manifold whose boundary is a torus. We choose a particular simple closed curve $\alpha$ on $\partial(\Sigma_2(B,\tau(T)))$ called a \textit{slope}. The Dehn filling operation consists in gluing a solid torus $V$ to $\Sigma_2(B,\tau(T))$ by identifying the meridian curve of $\partial V$ with $\alpha$. The obtained $3$-manifold is denoted by $\Sigma_2(B,\tau(T))(\alpha)$. We use the Monesinos trick which gives a presentation of that manifold as the double branched covering of $S^3$, the branched set of which is obtained by attaching a rational tangle to $\tau(T)$ in a prescribed way \cite{Montesinos}, and then we show the main following result.

\begin{theo}
 If $T$ is a connected alternating tangle and if $(B,\tau(T))$ is its alternating encirclement, then for infinitely many slopes $\alpha$ on the torus $\partial(\Sigma_2(B,\tau(T)))$, the manifolds $\Sigma_2(B,\tau(T))(\alpha)$ are $L$-spaces with non-left-orderable fundamental groups. Moreover, several of these manifolds are non-Seifert fibered.
\end{theo}

We will give more detailed statements in Paragraph \ref{Main_results}.\\
This paper is organized as follows. In the second section we give a brief overview of the main tools needed in the paper: Tangles, rational tangles, Montesinos links, quasi-alternating links, Dehn fillings, Montesinos trick and the Coarse Brunner's presentation. Then  we state our main results in the second section and give some applications. The third section is devoted to proofs of the main theorems. At the end of the paper, we ask two interesting questions raised by some of our results.\\

AKNOWLEDGEMENTS. We would like to thank Michel Boileau for his advice and helpful feedback. Thanks also to the referee for his/ her valuable comments and suggestions which allowed to improve this paper.

\section{Preliminaries}

\subsection{Tangles}
In this paper, we call a \textit{tangle} $T$ any pair $(B,A)$ where $B$ is a $3$-ball and $A$ is properly embedded $1$-dimensional manifold in $B$ and which meets the boundary of $B$ in four distinct points. Two tangles $T$ and $T^\prime$ are \textit{equivalent} if there is an ambient isotopy of the $3$-ball which is the identity on the boundary and which takes $T$ to $T^\prime$.\\
We assume that the four endpoints lie in the great circle of the boundary sphere of a $3$-ball $B^3$ which joins the two poles. That great circle bounds a two disk $B^2$ in $B^3$. We consider a regular projection of $B^3$ on $B^2$. The image of a tangle $T$ by that projection in which the height information is added at each of the double points is called a \textit{tangle diagram} of $T$. Two tangle diagrams will be equivalent if they are related by a finite sequence of planar isotopies and Reidemeister moves in the interior of the \textit{projection disk} $B^2$. Two tangles will be \textit{equivalent} iff they have equivalent diagrams.\\
Depending on the context we will denote by $T$ the tangle or its projection.

The four endpoints of the arcs in the diagram are usually labeled $NW_T$,$NE_T$,$SE_T$, and $SW_T$ with symbols referring to the compass directions as in the Fig. \ref{figfig2}.

\begin{figure}[H]
\centering
\includegraphics[width=0.2\linewidth]{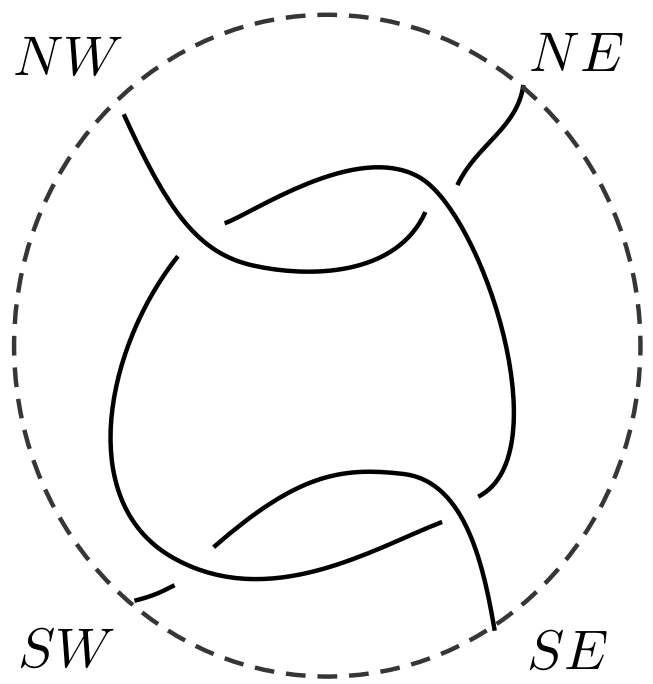}   
\caption{An alternating tangle diagram $T$.}
\label{figfig2}
\end{figure}

A tangle diagram $T$ is said to be \textit{disconnected} if either there exists a simple closed curve embedded in the projection disk, called a \textit{splitting loop}, which do not meet $T$, but encircles a part of it, or there exists a simple arc properly embedded in the projection disk, called a \textit{splitting arc}, which do not meet $T$ and splits the projection disk into two disks each one containing a part of $T$. A tangle diagram is \textit{connected} if it is not disconnected.

A tangle diagram $T$ is said to be \textit{locally knotted} if there exists a simple closed curve $C$ embedded in the interior of the disk projection, called a \textit{factorizing circle} of $T$, which meets $T$ transversally at two points and bounds a disk inside the disk projection which meets $T$ in a knotted spanning arc.

We adopt the notations used for rational tangles by Goldman and Kauffman in \cite{goldman} and Kauffman and Lambroupoulou in \cite{kauffman}. In Fig. \ref{fig4}, we recall some operations defined on tangles.

\begin{figure}[H]
\centering
\includegraphics[width=0.7\linewidth]{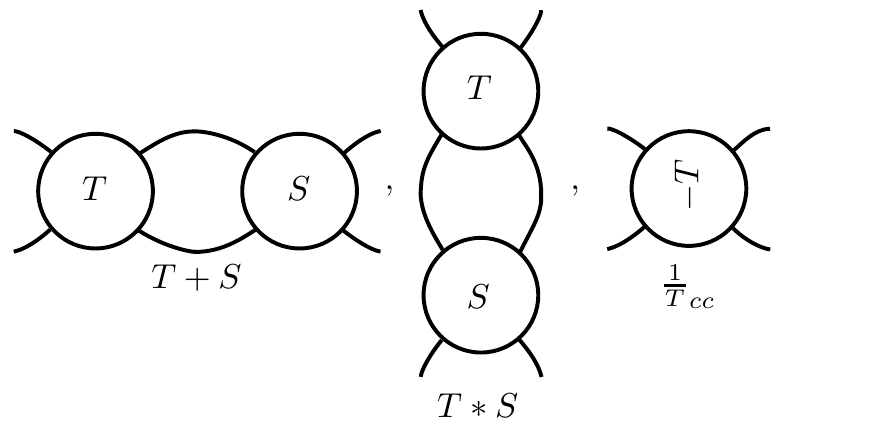}   
\caption{Some operations on tangle diagrams.}
\end{figure}

A ($-\pi$) rotation of a tangle diagram $T$ in the horizontal (respt. vertical) axis is called \textit{horizontal Flip} (respt. \textit{vertical Flip} and will be denoted by $T_h$ (respt. $T_v$). That is the tangle diagram obtained by rotating the ball containing $T$ in space around the horizontal (respt. vertical) axis as shown in Fig. \ref{figfig4bis} and then project the new tangle by the same projection function as that used to get $T$. Note that if $T$ is an alternating tangle diagram, then $T_h$ is also alternating. Note that the Flip operation preseves the isotopy class of a rational tangle (Flip Theorem 1. \cite{goldman}).

\begin{figure}[H]
\centering
\includegraphics[width=0.7\linewidth]{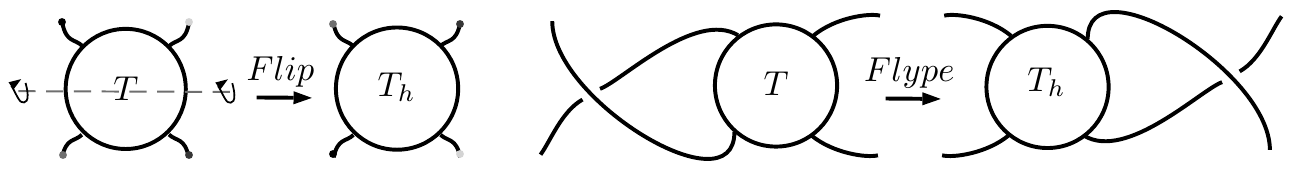}
\caption{Flip and Flype moves.}
\label{figfig4bis}
\end{figure}

A \textit{Flype} is an isotopy of tangles that is depicted by the Fig. \ref{figfig4bis}.
 
A tangle diagram $T$ provides two link diagrams: the \textit{Numerator} of $T$, denoted by $n(T)$, which is obtained by joining with simple arcs the two upper endpoints $(NW_T,NE_T)$ and the two lower endpoints $(SW_T,SE_T)$ of $T$, and the \textit{Denominator} of $T$, denoted by $d(T)$, which is obtained by joining with simple arcs each pair of the corresponding top and bottom endpoints $(NW_T,SW_T)$ and $(NE_T,SE_T)$ of $T$ (see Fig. \ref{figfig5}). We denote $N(T)$ and $D(T)$ respectively the corresponding links. We denote $N(T)$ and $D(T)$ respectively the corresponding links. We also denote by $N_T$ and $D_T$ the respective determinants of the links $N(T)$ and $D(T)$.

\begin{figure}[H]
\centering
\includegraphics[width=0.5\linewidth]{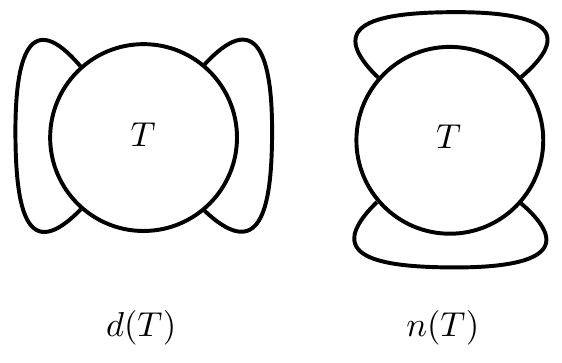}   
\caption{The denominator and the numerator of a tangle diagram $T$.}
\label{figfig5}
\end{figure}

A tangle diagram $T$ is called \textit{alternating} if the ``over'' or ``under'' nature of the crossings alternates as one moves along any arc of $T$. A tangle is said to be \textit{alternating} if it admits an alternating diagram. If $T$ is a connected alternating tangle diagram such that the link diagrams $n(T)$ and $d(T)$ are both non-split and reduced, then $T$ is said to be a \textit{strongly alternating diagram}. 

Let $T$ be an alternating connected tangle diagram. Consider the arc of $T$ which have $NW_T$ as an endpoint. Suppose that when we move along that arc starting at $NW_T$ we pass below at the first encountered crossing. Then the arc of $T$ which ends at the point $SE_T$ will also pass below at the last encountered crossing before reaching $SE_T$ and the arc of $T$ which starts at $NE_T$ will pass over at the first encountered crossing. It is easy to see that the arcs of $T$ coming from diametrically opposite endpoints both pass over or below at the first encountered crossing. That remark enables us to distinguish two types of alternating connected tangle diagrams which we call \textit{type 1} tangles and \textit{type 2} tangles as shown in the Fig. \ref{figfig7}. 

\begin{figure}[H]
\centering
\includegraphics[width=0.4\linewidth]{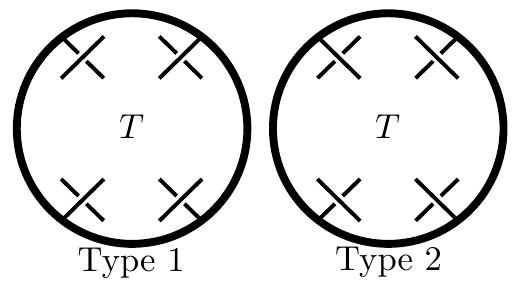}   
\caption{Type 1 and Type 2 alternating tangle diagrams.}
\label{figfig7}
\end{figure}

In order to achieve our particular $3$-manifolds, we will use the tangles obtained as follows. Let $T$ be a connected alternating tangle. We call the \textit{alternating encirclement} of $T$ denoted by $\tau(T)$, the tangle $T$ encircled by a trivial closed curve as depicted in Fig.\ref{ae} such that the resulting tangle is alternating. Note that the notion of alternating encircled tangles first appeared in \cite{thistlethwaite}. If $T$ is of type $2$, then $\tau(T)$ is a connected alternating tangle of type 1. In what follows, we will assume that $T$ is of type $2$.

\begin{figure}[H]
\centering
\includegraphics[width=0.3\linewidth]{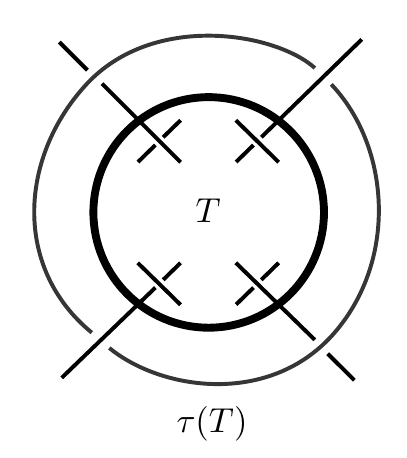}   
\caption{The alternating encirclement of the tangle $T$.}
\label{ae}
\end{figure}

\subsection{Rational tangles}

A \textit{rational tangle} $t$ is a tangle in $B^3$ such that the pair $(B^3, t)$ is homeomorphic to $(B^2 \times [0,1] , \left\lbrace x,y \right\rbrace \times [0,1]) $, where $x$ and $y$ are points in the interior of $B^2$. The elementary rational tangle diagrams $0$, $\pm1$, $\infty$ are shown in Fig. \ref{figfig8}.
\begin{figure}[H]
\centering
\includegraphics[width=0.25\linewidth]{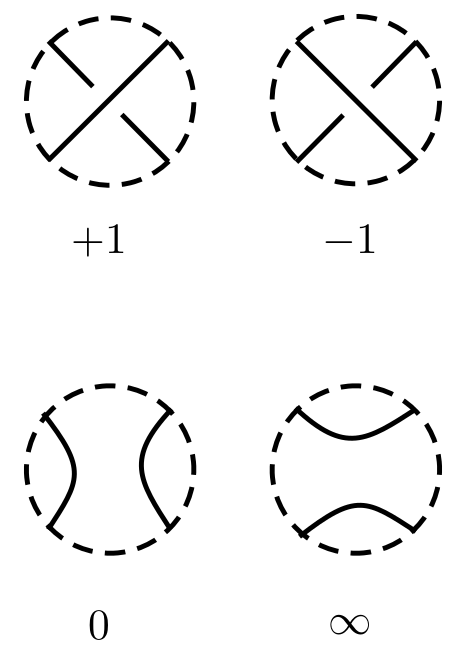}  
\caption{Elementary rational tangles.}
\label{figfig8}
\end{figure}
The sum of $n$ copies of the tangle diagram $1$ or of $n$ copies of the tangle $-1$ are respectively the \textit{integer tangle diagrams} denoted also by $n$ and $-n$.
If $t$ is a rational tangle diagram then $\dfrac{1}{t}_c$ and $\dfrac{1}{t}_{cc}$ are equivalent and both represent the \textit{inversion} of $t$ denoted by $\dfrac{1}{t}$.

 Let $t$ be a rational tangle diagram and $p,q \in \mathbb{Z}$, we have the following equivalences:
$$  p+t+q = t+p+q \text{  ,  } \frac{1}{p} * t * \frac{1}{q} = t * \frac{1}{p+q}.$$
$$ t * \frac{1}{p} = \frac{1}{p+\frac{1}{t}} \text{  ,  } \frac{1}{p} * t = \frac{1}{\frac{1}{t}+p}. $$

Using the above notations and equivalences one can naturally associate to any continued fraction
$$\displaystyle a_1 + \frac{1}{ a_2  + \frac{1}{  \ddots +\frac{1}{ a_{n-1}  + \frac{1}{ a_n } }}},\,\ a_i\in\mathbb{Z},$$ a tangle diagram as shown in Fig. \ref{figfig9} denoted by $\left[ a_1,...,a_{n} \right]$. 

\begin{figure}[H]
\centering
\includegraphics[width=0.8\linewidth]{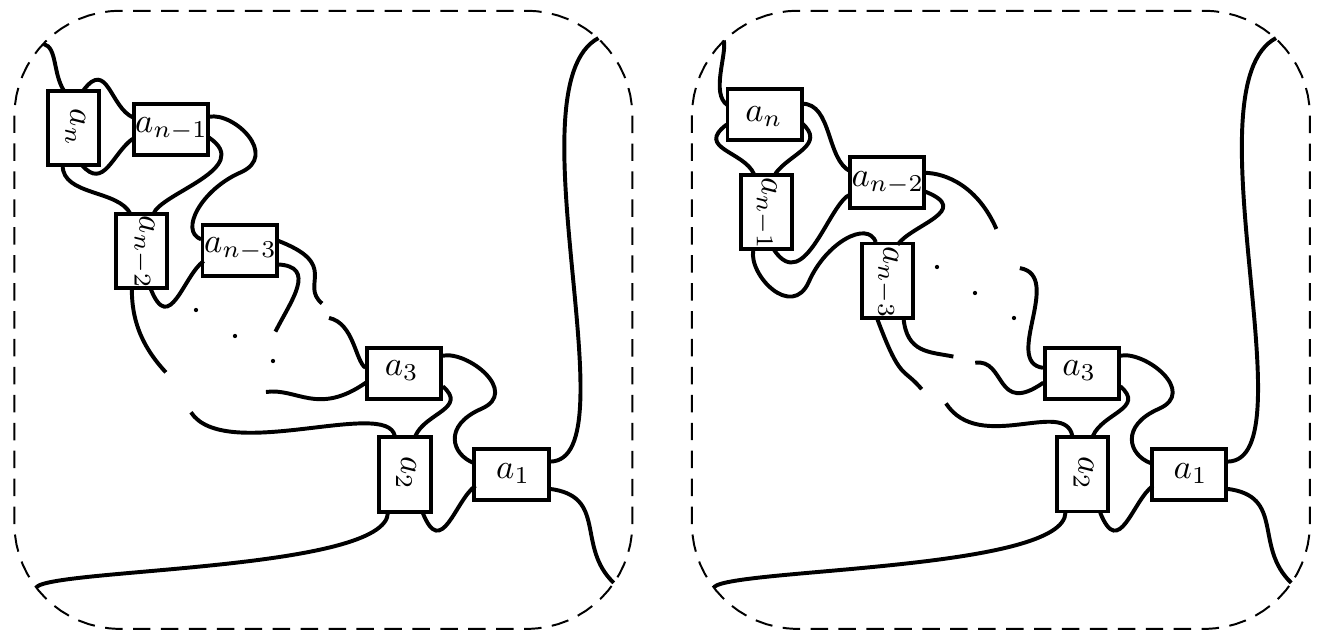}  
\caption{The standard rational tangle diagram $\left[ a_1,...,a_{n} \right]$ according to $n$ is even (left) or odd (right).}
\label{figfig9}
\end{figure}

Conversely, it is known that for any rational tangle $t$, there exists an integer $n \geq 1$ and integers $ a_1 \in \mathbb{Z}$, $a_2, ... , a_n \in \mathbb{Z}\setminus\lbrace 0 \rbrace,$ all of the same sign, such that  $t= \left[ a_1,...,a_{n} \right]$. Then $t$ corresponds to a continued fraction and then to a rational number called the fraction of the tangle.

J. H. Conway showed in \cite{conway} that two rational tangles are equivalent if and only if they have the same fraction. Then any rational tangle $t$ can be represented by a continued fraction $\left[ a_1  , ... ,  a_n \right] = \frac{a}{b}$ where $a$ and $b$ are two coprime integers. 

The \textit{standard diagram} of a rational tangle $t$ will be the connected alternating diagram naturally associated to the continued fraction of $t$ described above. In what follows a rational tangle diagram will mean the standard one.\\

An \textit{algebraic tangle} is a tangle obtained from rational tangles by a sequence of $+$ and $*$ operations.

\subsection{Montesinos links}
Let $t_i \neq 0 , \pm 1$, for $i \in [\![1,n]\!]$, be rational numbers, and let $e$ be an integer. A \textit{Montesinos link} is defined as $ M(e;t_1,...,t_n) := N(e+ \frac{1}{t_1} + ... + \frac{1}{t_n})$. Those links were introduced by Montesinos in \cite{montesinos2}. 

Let $t = \frac{\alpha}{\beta} $  be a rational number with $\beta > 0$. The \textit{floor} of $t$ is $ \lfloor t \rfloor = max \left\lbrace x \in \mathbb{Z} / x \leq t  \right\rbrace ,$ and the \textit{fractional part} of $t$ is $ \left\lbrace t \right\rbrace = t - \lfloor t \rfloor < 1.$ For $t \neq 1$, define $\hat{t} = \frac{1}{ \left\lbrace \frac{1}{t} \right\rbrace } > 1.$ We also put 
$\left( \frac{\alpha}{\beta} \right)^f = \begin{cases}
\frac{\alpha}{\beta - \alpha} &\text{ if } \frac{\alpha}{\beta} > 0 \\
\frac{\alpha}{\beta + \alpha} &\text{ if } \frac{\alpha}{\beta} < 0
\end{cases}$

Let $L$ be the Montesinos link $M(e;t_1,...,t_n)$. We define $ \varepsilon(L) = e + \displaystyle \sum_{i=1}^n \lfloor \frac{1}{t_i} \rfloor$. The link $L$ is isotopic to $M( \varepsilon(L) ; \hat{t_1}, ... , \hat{t_n})$ (Proposition 3.2 , \cite{champanerkar2}). The link $M( \varepsilon(L) ; \hat{t_1}, ... , \hat{t_n})$ is called the \textit{reduced form} of the Montesinos link $L = M(e;t_1,...,t_n)$.

The double branched covering of the $3$-sphere branched over a Montesinos link is a Seifert fibered space as shown in \cite{montesinos2}.

\subsection{Quasi-alternating links}
A link diagram is \textit{alternating} if the over or under nature of the crossings alternates along every link-component in the diagram: the crossings go ``...over, under, over, under,...'' when considered from any starting point. A link is said to be \textit{alternating}  if it possesses such a diagram.

The set of quasi-alternating links appeared in the context of link homology as a natural generalization of alternating links. They were defined in \cite{ozsvath2} by Ozsv{\'a}th and Szab{\'o}. In \cite{manolescu}, Manolescu and Ozsv{\'a}th showed that quasi-alternating links are homologically thin for both Khovanov homology and knot Floer homology as alternating links with which they share many properties. On the other hand, it was shown in \cite{ozsvath2} that every non-split alternating link is quasi-alternating and that the double branched covering of any quasi-alternating link is an $L$-space. Recall that a link $L$ is \textit{non-split} if there is no $2$-sphere in the complement of $L$ in $S^3$ that separates some components of $L$ from the others, and that a link diagram $D$ is \textit{non-split} if there exists no simple closed curve in the plane that separates some components of $D$ from the others. \\
If $D$ is a link diagram, we denote by $\mathcal{L}(D)$ the link for which $D$ is a projection. Quasi-alternating links are defined recursively as follows:
\begin{mydef}
 The set $\mathcal{Q}$ of \textbf{quasi-alternating links} is the smallest set of links satisfying the following properties:
\begin{enumerate}
\item The unknot belongs to $\mathcal{Q}$,
\item If $L$ is a link with a diagram $D$ containing a crossing $c$ such that 
\begin{enumerate}
\item for both smoothings of the diagram $D$ at the crossing $c$ denoted by $D^c_0$ and $D^c_\infty$ as in figure \ref{fig.1}), the links $\mathcal{L}(D^{c}_{0})$ and $\mathcal{L}(D^{c}_{\infty})$ are in $\mathcal{Q}$ and,
\item $\det(L)=\det(\mathcal{L}(D^{c}_{0}))+\det(\mathcal{L}(D^{c}_{\infty})).$
\end{enumerate}
Then $L$ is in $\mathcal{Q}$. In this case we will say that $c$ is a
\textit{quasi-alternating crossing} of $D$ and that $D$ is quasi-alternating at $c$.
\end{enumerate}

\end{mydef}
\begin{figure}[H]
\centering
\includegraphics[width=0.4\linewidth]{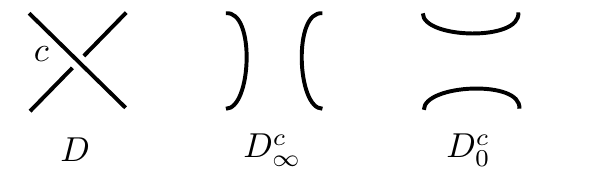}
\caption{The link diagram $D$ and its smoothings $D^{c}_{0}$ and $D^{c}_{\infty}$ at the crossing $c$.}
\label{fig.1}
\end{figure}

\begin{rmk}
A non-split alternating link diagram is quasi-alternating at each non-nugatory crossing by Lemma 3.2 in \cite{ozsvath2}. So, we can compute the determinant of a non-split alternating link by performing successive smoothings at the non-nugatory crossings and then by adding the determinants of the links produced at each step until we get the trivial knot. We will use this remark in our computations.
\label{rmk2}
\end{rmk}

\subsection{Dehn fillings}
Let $\alpha$ and $\beta$ be unoriented simple closed curves on a torus $T^2$. Then $\alpha$ and $\beta$ are isotopic if and only if $\left[ \alpha \right] = \pm \left[ \beta \right] \in H_1(T^2)$. A \textit{slope} on $T^2$ is an isotopy class of unoriented essential simple closed curves on $T^2$. Recall that $T^2$ bounds a solid torus $V$. A \textit{meridian} $m$ of $T^2$ is an unoriented essential simple closed curve on $T^2$ that bounds a disk in $V$. Note that the meridian is unique up to isotopy. A \textit{longitude} $l$ of $T^2$ is an unoriented essential simple closed curve on $T^2$ that meets the meridian tranversally at a single point. The pair $(m,l)$ provides a basis of $H_1(T^2) \cong \mathbb{Z} \oplus \mathbb{Z}$. More precisely, if $\alpha$ is a slope on $T^2$, then $\left[ \alpha \right] = \pm (a \left[ m \right] + b \left[ l \right]) \in H_1(T^2)$ for some coprime integers $a$ and $b$. The correspondence $\alpha \leftrightarrow \frac{a}{b} $ is one to one. This establishes an identification of the slopes on $T^2$ with the set $\mathbb{Q} \cup \left\lbrace\frac{1}{0} \right\rbrace$. 

Let $K$ be a knot in $S^3$. Let $V$ be a tubular neighborhood of $K$. Let $m$ be the meridian of $\partial V$. We choose a longitude $l$ of $\partial V$ to be the trace of a Seifert surface of $K$ on $\partial V$. So $l$ is null-homologous in the exterior of $K$. Recall that the choice of such longitude is unique up to ambient isotopy.\\
Recall that the \textit{surgery} on $S^3$ along $K$ with slope $\dfrac{a}{b}$, $a,b\in\Z$, is the operation which consists in removing the interior of $V$ and then gluing a solid torus $S^1 \times B^2$ to $S^3 \setminus \accentset{\circ}{V}$ such that the meridian $(* \times \partial B^2)$ of $S^1 \times B^2$ is identified with the slope $ \left[ \alpha \right] = a \left[ m \right] + b \left[ l \right]$. A surgery with an integer slope is said to be an \textit{integer} surgery.

Let $M$ be a $3$-manifold with torus boundary $T_0$ and $\alpha$ be a slope on $T_0$. Define the $\alpha$-\textit{Dehn filling} of $M$ denoted by $M(\alpha)$, to be the manifold obtained by gluing a solid torus $V$ to $M$ so that the boundary of the meridional disk in $V$ is glued to $\alpha$:
$$ M(\alpha) = M \cup_{T_0=\partial V} V.$$
More details about Dehn fillings and surgery can be found in Gordon's paper \cite{Gordon} and in the book of Prasolov and Sossinsky \cite{prasolov}.

\subsection{The Montesinos trick}
Let $T=(B,A)$ be a tangle and $\Sigma_2(B,A)$ the double branched covering of $B$ along $A$. Notice that $\Sigma_2(B,A)$ is a $3$-manifold with torus boundary. Let $(\gamma_{\infty},\gamma_0)$ be a pair of embedded arcs in $\partial B$ with endpoints on $\partial A$ as shown in Fig. \ref{pfig}. The pair $(\gamma_{\infty},\gamma_0)$ lifts to a (unoriented) basis $(\tilde{\gamma}_{\infty},\tilde{\gamma}_0)$ for $H_1(\partial \Sigma_2(B,A), \mathbb{Z})$. By fixing an orientation so that $\tilde{\gamma}_{\infty}.\tilde{\gamma}_0 = 1$, we obtain a basis to do Dehn fillings of $\Sigma_2(B,A)$ called the \textit{standard} basis. Montesinos observed in \cite{Montesinos} that a Dehn filling of $\Sigma_2(B,A)$ may be viewed as a double branched covering of $S^3$ along a specified link. More precisely, for a given slope $\alpha = p\tilde{\gamma}_{\infty} + q\tilde{\gamma}_0$ in $\partial \Sigma_2(B,A)$, Montesinos showed that $\Sigma_2(B,A)(\alpha) \cong \Sigma_2(S^3,N(-\frac{p}{q}+T))$, where $\Sigma_2(S^3,N(-\frac{p}{q}+T))$ is the double branched covering of $S^3$ along $N(-\frac{p}{q}+T)$. This observation is referred to as the \textit{Montesinos trick}. For the seek of simplicity, we denote by $\Sigma_2(B,A)(\frac{p}{q})$ the manifold $\Sigma_2(B,A)(\alpha)$ when $\alpha$ is the slope corresponding to the fraction $\frac{p}{q}$ with respect to the standard basis for Dehn fillings of $\Sigma_2(B,A)$. 

\begin{figure}[H]
\centering
\includegraphics[width=0.25\linewidth]{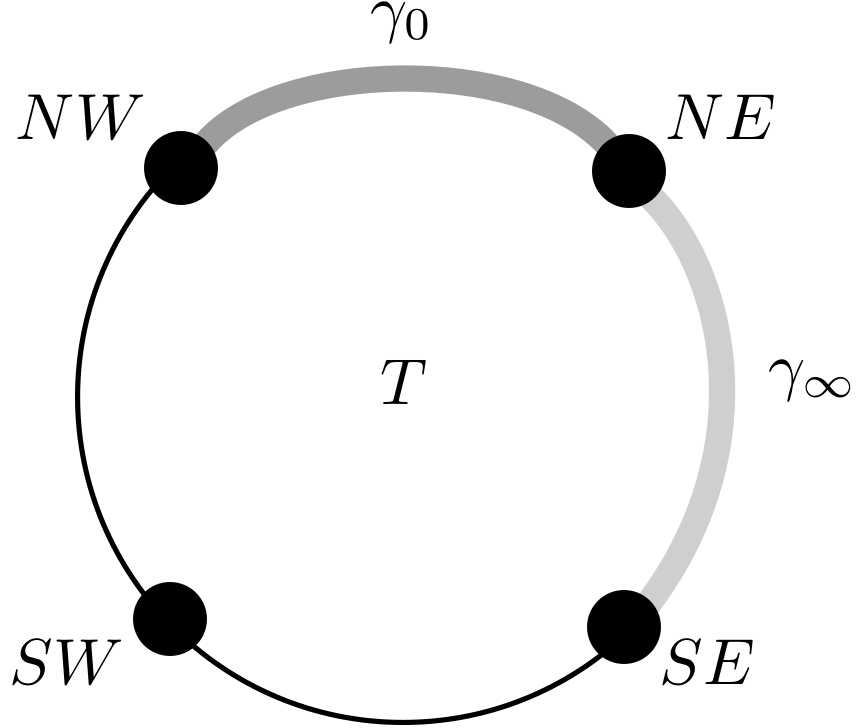}
\caption{The pair of curves $(\gamma_{\infty},\gamma_0)$.}
\label{pfig}
\end{figure}

\noindent\textbf{Band surgery.} Let $L$ be a link in $S^3$ and $b: I \times I \rightarrow S^3$ an embedding such that $L \cap b(I \times I)=b(\partial I \times I)$, where $I$ is the unit interval. Let $L'$ denote the link obtained by replacing $b(\partial I \times I)$ in $L$ by $b(I \times \partial I)$. Then we say the link $L'$ results from \textit{band surgery} along $L$.\\
If $L$ is a link obtained by a band surgery along the trivial knot $U$, then we can describe a surgery on $S^3$ that provides the double branched covering of the link $L$ as follows.\\ 
Let $\alpha$ denote the simple arc $b(I \times \left\lbrace \frac{1}{2} \right\rbrace)$. Note that $\alpha$ is embedded in $S^3$ with endpoints on $U$. Let $B$ be a regular neighborhood of $\alpha$ in $S^3$. Without loss of generality, one can assume that $ b(I \times I) \subset B$ and $\partial B \cap b(I \times I) = b( \left\lbrace (0,0),(1,0),(0,1),(1,1) \right\rbrace )$, meaning that the band $b( I \times I)$ is entirely contained in $B$ and meets the boundary of $B$ only at its four corners (see Fig. \ref{ballband}). Note that the pair $(S^3 \setminus \accentset{\circ}{B}, (S^3 \setminus \accentset{\circ}{B})\cap U)$ is a tangle and the pairs $(B, B \cap b(I \times \partial I))$ and $(B,B \cap U)$ are rational tangles. By using the Montesinos trick, we have that $\Sigma_2(S^3 \setminus \accentset{\circ}{B}, (S^3\setminus \accentset{\circ}{B}) \cap U)(\frac{p}{q})$ is homeomorphic to $\Sigma_2(S^3,L)$, where $-\frac{p}{q}$ is the fraction of the rational tangle $(B, B \cap b(I \times \partial I))$. This is equivalent to say that the manifold $\Sigma_2(S^3,L)$ is obtained by a surgery on $\Sigma_2(S^3,U) \cong S^3$ along a lift of the arc $\alpha$ (which is a knot in $S^3$). 

\begin{figure}[H]
\centering
\includegraphics[width=0.25\linewidth]{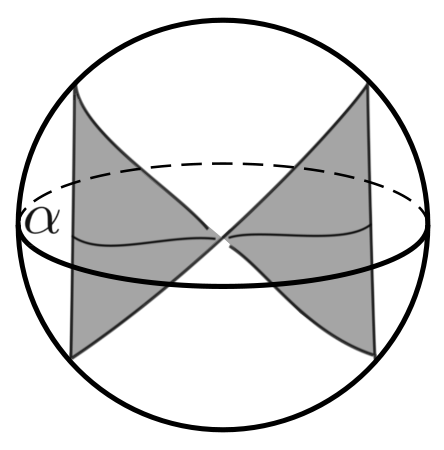}
\caption{An embedded band in the $3$-ball $B$.}
\label{ballband}
\end{figure}

\subsection{Coarse Brunner's presentation}
In this section we recall the construction of the coarse Brunner's presentation and the non-left-orderability criterion based on that presentation. For more details we refer to Paragraph 3 in \cite{ito}.
\subsubsection{Brunner's presentation}
Let $D$ be a diagram of a non-split link $L$. We consider a checkerboard coloring of $D$  with the convention that the unbounded region is not colored. Then we get a surface, possibly non-orientable, whose boundary is the link $L$. We call the obtained surface a \textit{checkerboard surface}.\\
\noindent\textit{\textbf{Decomposition graph.}} The checkerboard surface is decomposed as a union of disks and twisted bands in an obvious way. Among these decompositions we choose the maximal one, that is the disc-twisted band decomposition having the minimal number of twisted bands. The obtained decomposition is called a \textit{disk-band decomposition} of the checkerboard surface. To such a decomposition we associate a labeled planar graph $G_D$, called the \textit{decomposition graph}, as follows: we assign a vertex to each disk, and to each twisted band that connects two disks we assign a labeled edge connecting the corresponding vertices. The labeling of edges is done according to Fig.\ref{figbis3}. A component of $\R^2\setminus G_D$ is called a \textit{region} of the diagram $D$. A region of $D$ is identified with a white colored region of the diagram $D$.

\begin{figure}[H]
\centering
\includegraphics[width=0.7\linewidth]{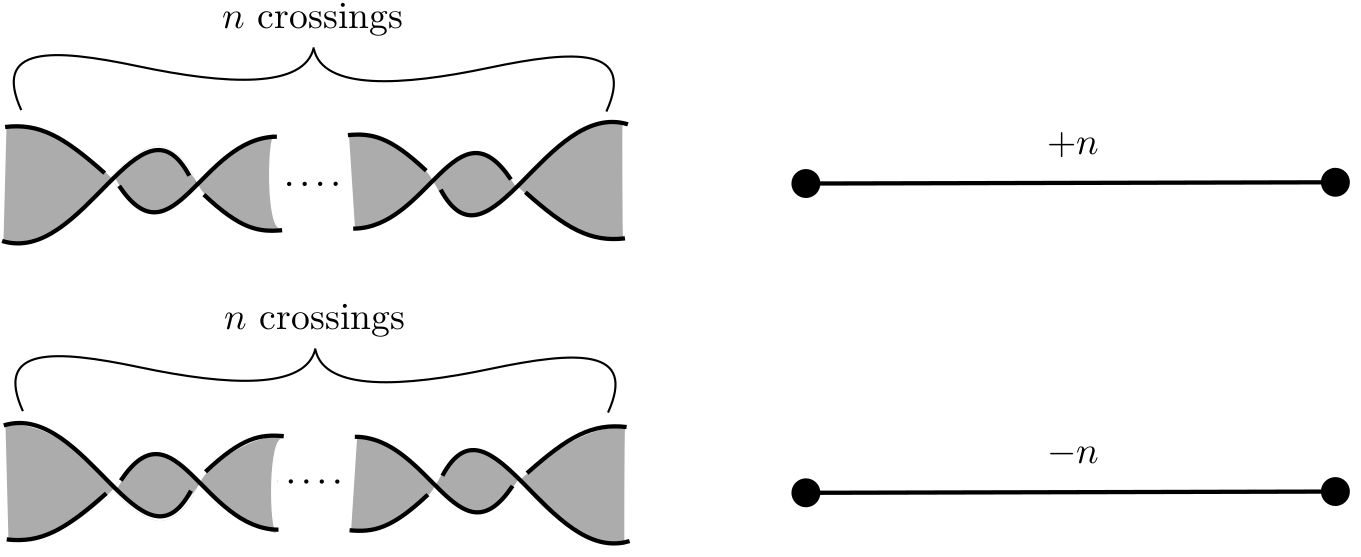}
\caption{Assignment of labeled edges to twisted bands.}
\label{figbis3}
\end{figure}

\noindent\textit{\textbf{Connectivity graph.}} From the decomposition graph, we construct an oriented planar graph called the \textit{connectivity graph}, denoted by $\tilde{G}_D$,  as follows: The vertices of $\tilde{G}_D$ are the same as those of $G_D$, while an edge is obtained by connecting two vertices (discs) by a single arc corresponding to one twisted band connecting them. Explicitly, this amounts to connect two vertices by choosing one of the edges connecting the two vertices in $G_D$.Then we orient the edges of $\tilde{G}_D$ according to the rule shown in Fig.\ref{figbis4}. We endow the decomposition graph $G_D$ with the orientation induced by that of the connectivity graph $\tilde{G}_D$ in the obvious way.

\begin{figure}[H]
\centering
\includegraphics[width=0.4\linewidth]{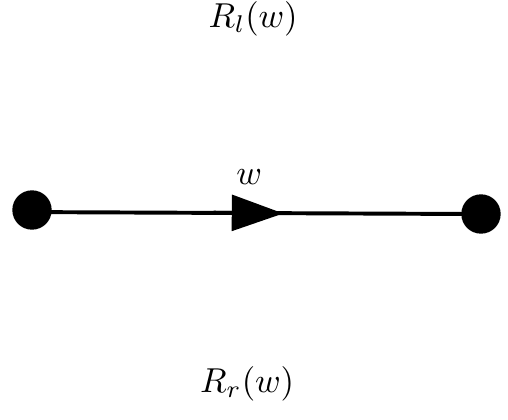}
\caption{The left-adjacent and right-adjacent regions of an oriented edge.}
\label{figbis4}
\end{figure}

For an edge $w$ of $G_D$, we distinguish two regions of $D$, the \textit{left-adjacent region} $R_l(w)$ and the \textit{right-adjacent region} $R_r(w)$, as shown in Fig.\ref{figbis4}. 

By using these notions, Brunner's presentation of $\pi_1(\Sigma_2(S^3,L))$ is given as follows \cite{brunner}:
\begin{theo} Let $L$ be a non-split link in $S^3$ represented by a diagram $D$, and $G_D$ and $\tilde{G}_D$ be the decomposition and the connectivity graphs. Then the fundamental group of $\Sigma_2(S^3,L)$ has the following presentation called Brunner's presentation.\\
\textbf{[Generators]} 

\textbf{Edge generators:} The edges $\left\lbrace W_i \right\rbrace$ of the connectivity graph $\tilde{G}_D$. 

\textbf{Region generators:} The regions $\left\lbrace R_j \right\rbrace$ of the link diagram $D$.\\
\textbf{[Relations]} 

\textbf{Local edge relations:} $W=(R_l(w)^{-1} R_r(w))^a$, where $w$ is an edge of the decomposition graph $G_D$ with label $a$, and $W$ is an edge generator corresponding to $w$. 

\textbf{Global cycle relations:} $W_n^{\pm1}...W_1^{\pm1}=1$, if the edge-path $W_n^{\pm1}...W_1^{\pm1}$ represents a loop in $\mathbb{R}^2$. 

\textbf{Vanishing relation:} $R_0=1$, where $R_0$ is the unbounded region generator.
\end{theo}

Here, the edge $W^{-1}$ is the edge $W$ with the opposite orientation. Also, we use the convention that $W_2W_1$ is representing the edge-path that travels along $W_1$ first, then along $W_2$.
\begin{empl}
Let $D$ be the link diagram on the left in Fig. \ref{bpe}. We construct the decomposition and connectivity graphs as shown in the same figure. The Brunner's presentation of the group $\pi_1(\Sigma_2(S^3,\mathcal{L}(D)))$ is written as follows.
\[\Biggl\langle 
       \begin{array}{l|cl}
                        & W_1 =R_1^3,W_2=R_2^{-1}R_1,W_3=R_4^{-1}R_1,W_4=R_2,W_5=R_4, \\
                        &  W_6=(R_3^{-1}R_2)^{2}=(R_4^{-1}R_2)^{2} \\
           \left\lbrace W_i \right\rbrace_{1 \leq i \leq 6} , \left\lbrace R_j \right\rbrace_{1 \leq j \leq 4} &  & \\
                        & W_1^{-1}W_2W_3=W_2^{-1}W_4W_6=W_5^{-1}W_6W_3=1  &                                             
        \end{array}
     \Biggr\rangle\]
\end{empl}

\begin{figure}[H]
\centering
\includegraphics[width=0.8\linewidth]{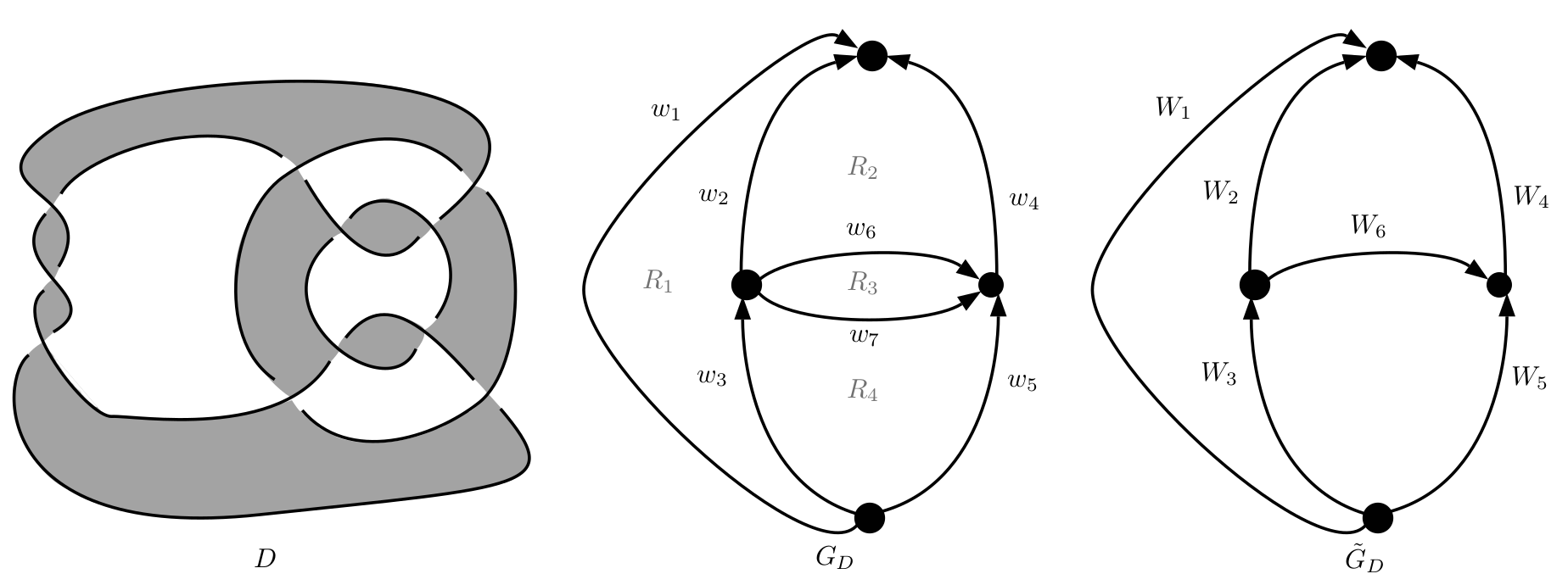}
\caption{A checkerboard surface of a link diagram $D$ (left), the associated decomposition graph $G_D$ (middle), and the associated connectivity graph $\tilde{G}_D$ (right).}
\label{bpe}
\end{figure}
\subsubsection{Universal ranges}
Let $G$ be a group, $h \in G$, and $<_G$ be a left-order on $G$. For any rational numbers $P=\frac{a}{b},Q=\frac{c}{d}$ such that $\frac{a}{b} \leq \frac{c}{d}$ and $b,d > 0$, let $ \llbracket \frac{a}{b} ~,~ \frac{c}{d} \rrbracket_{h,<_G} $ be a subset of $G$ defined by

\[
  \llbracket P ~,~ Q \rrbracket_{h,<_G} =
  \begin{cases}
                                   \left\lbrace g \in G | gh=hg, h^{a} \leq_G g^{b}, g^d \leq_G h^c \right\rbrace & \text{if } h \geq_G 1 
                                    \\
                                   \left\lbrace g \in G | gh=hg, g^b \leq_G h^a, h^c \leq_G g^d \right\rbrace & \text{if } h \leq_G 1 
  \end{cases}
\]

Note that under the assumption that $b,d > 0$, and for any $m,n \in \mathbb{Z} \setminus \left\lbrace 0 \right\rbrace$, we have $$ \llbracket \frac{a}{b} ~,~ \frac{c}{d} \rrbracket_{h,<_G} =  \llbracket \frac{ma}{mb} ~,~ \frac{nc}{nd} \rrbracket_{h,<_G}.$$ 
So the set $\llbracket P ~,~ Q \rrbracket_{h,<_G}$ does not depend on the choice of the representatives of the rational numbers $P$ and $Q$. 

We define $$ \llbracket P ~,~ + \infty \rrbracket_{h,<_G} =  \bigcup_{Q > P} \llbracket P ~,~ Q \rrbracket_{h,<_G},$$
$$ \llbracket - \infty ~,~ Q \rrbracket_{h,<_G} =  \bigcup_{P < Q} \llbracket P ~,~ Q \rrbracket_{h,<_G},$$
$$ \llbracket - \infty ~,~ + \infty \rrbracket_{h,<_G} =  \bigcup_{P \in \mathbb{Z}_{>0}} \llbracket -P ~,~ P \rrbracket_{h,<_G}.$$

For $P \in \mathbb{Q} \cup \left\lbrace -\infty \right\rbrace$, $Q \in \mathbb{Q} \cup \left\lbrace +\infty \right\rbrace$, $P \leq Q$, define
$$\llbracket P ~,~ Q \rrbracket_{h} = \bigcap\limits_{<_G \in LO_G} \llbracket P ~,~ Q \rrbracket_{h,<_G},$$ where $LO_G$ is the set of all left-orders on $G$. If $g \in \llbracket P ~,~ Q \rrbracket_{h}$, then we say that $\llbracket P ~,~ Q \rrbracket$ is an \textit{$h$-universal range} of $g$.

\subsubsection{A left-orderability criterion}
Any link diagram $D$ can be decomposed into embedded algebraic tangles attached together with a set of strands. Such a decomposition of $D$ induces a decomposition of its checkerboard surface into a set of disks and subsurfaces corresponding to tangles. The last decomposition is called a \textit{tangle-strand decomposition} of $D$.\\
Let $S$ be a subsurface of the checkerboard surface of a link diagram $D$. Let $\Delta$ be the projection disk of the corresponding tangle. We consider the two arcs which constitute the intersection of $\partial\Delta$ with $S$. To distinguish the isotopy class of the tangle we will use, we agree that the endpoints of these two arcs will be labelled in the following way: one of them connects $NW$ to $SW$ while the other connects $NE$ to $SE$. We denote by $t$ the tangle corresponding to this labelling. Then the subsurface $S$ is denoted by $Q(t)$ and is called the \textit{tangle part} corresponding to $t$ (See elementary cases in Fig.\ref{figfig4}).

\begin{figure}[H]
\centering
\includegraphics[width=1\linewidth]{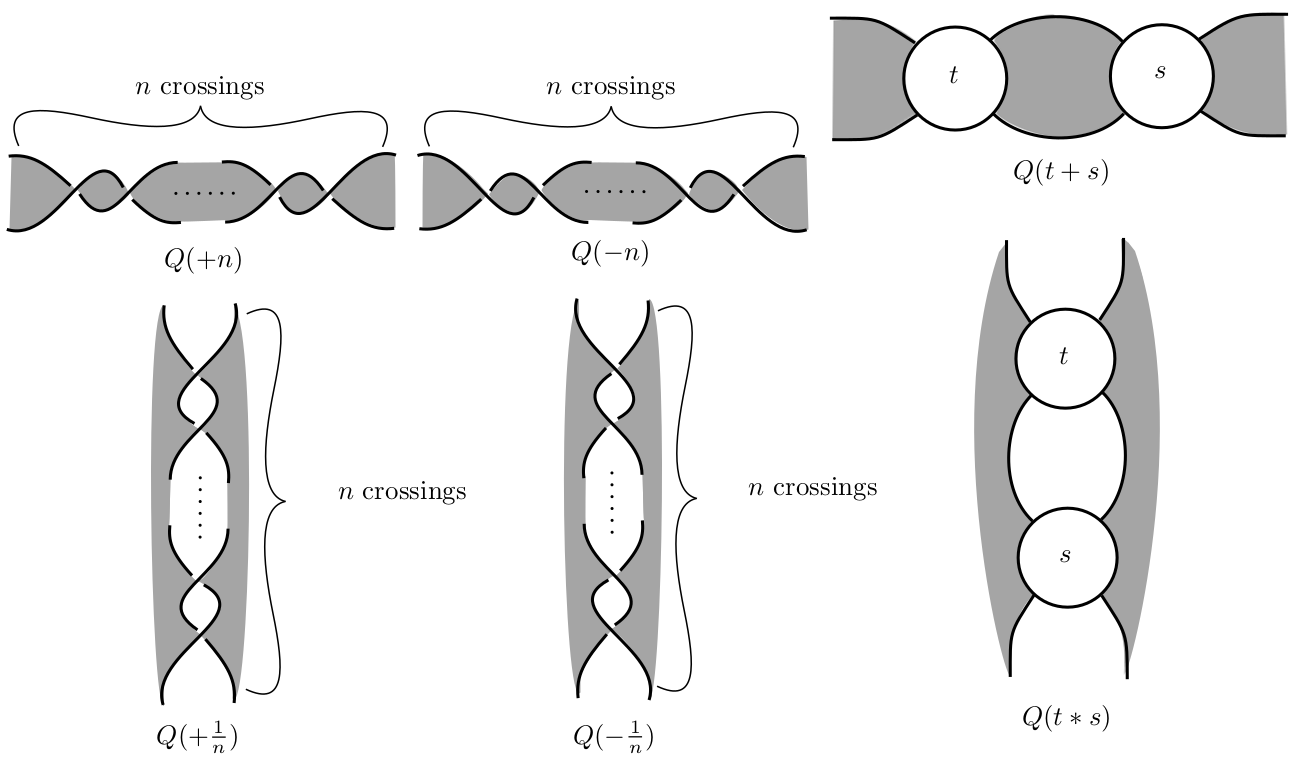}
\caption{Labeling of the tangle parts.}
\label{figfig4}
\end{figure}

\noindent\textbf{\textit{Coarse decomposition graph}.} To a tangle-strand decomposition of $D$, we associate an oriented planar graph $\Gamma_D$, called the \textit{coarse decomposition graph} of $D$ in the following way. The vertex of $\Gamma_D$ is a disk part of the tangle-strand decomposition. To each tangle part, we assign an edge connecting the vertices that correspond to the disks connected by the considered tangle part. We orient that edge according to the rule depicted in Fig.\ref{figbis2}.

\begin{figure}[H]
\centering
\includegraphics[width=0.7\linewidth]{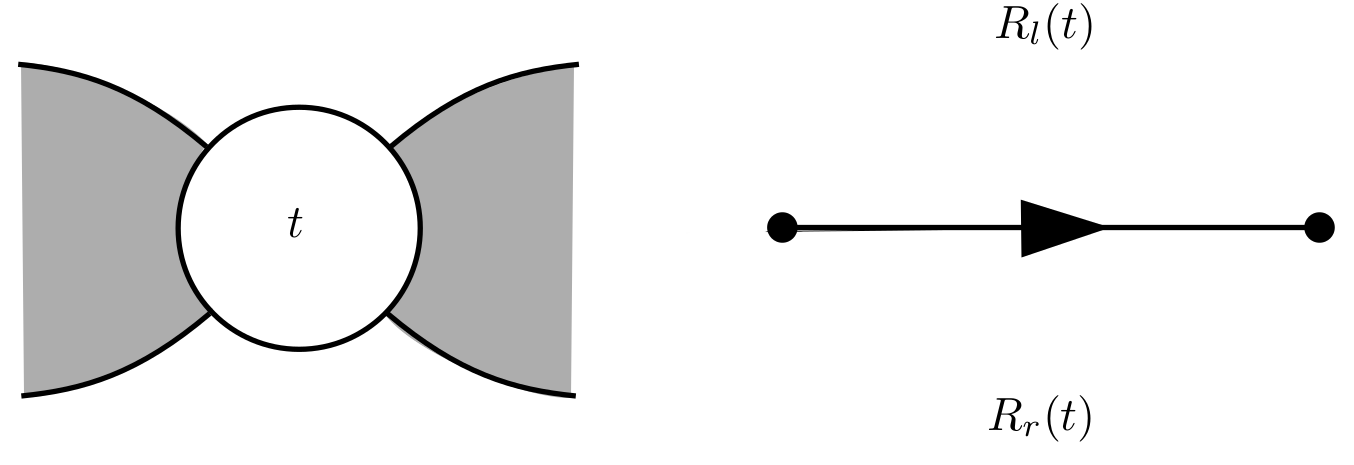}
\caption{Assignment of oriented edges to tangle parts.}
\label{figbis2}
\end{figure}

Let $Q(t)$ be a subsurface of the checkerboard surface as that introduced above. Denote by $\tilde{G}_t$ the subgraph of $\Gamma_D$ derived from the sub-diagram $t$ of $D$. Let $u$ and $v$ be the vertices corresponding to the disk parts joined by $Q(t)$ in the tangle-strand decomposition. We denote by $W_t \in \pi_1(\Sigma_2(S^3,\mathcal{L}(D)))$, the \textit{tangle element}  which is the uppermost edge-path in $\tilde{G}_t$ connecting the vertices $u$ and $v$. For convenience, the edge of $\Gamma_D$ that corresponds to $Q(t)$ is also denoted by $W_t$. Note that the regions of $\Gamma_D$ are elements of $\pi_1(\Sigma_2(S^3,\mathcal{L}(D)))$. For each edge $W_t$ of $\Gamma_D$, we distinguish two special regions: the left-adjacent region $R_l(t)$, and the right-adjacent region $R_r(t)$ as depicted in Fig.\ref{figbis2}. Ito showed that $W_t$ commutes with $R_l(t)^{-1}R_r(t)$ (Lemma 3.3, \cite{ito}).\\
A \textit{universal range of $t$} is an $(R_l(t)^{-1}R_r(t))$-universal range of $W_t$. A tangle-strand decomposition is said to be \textit{nice} if all tangles have universal range in $\llbracket - \infty ~,~ + \infty \rrbracket$. 
Now we are ready to define the coarse Brunner's presentation\\

Let $D$ be a link diagram representing a non-split link together with a coarse decomposition graph $\Gamma_D$ associated to a nice tangle-strand decomposition of $D$. The coarse Brunner's presentation $\mathcal{CB}$ of $D$ associated to $\Gamma_D$ is a set of generators and relations given as follows: 

\textbf{[Generators]} 

\textbf{Tangle generators:} The tangle elements $\left\lbrace W_i \right\rbrace$ (the edges of the coarse decomposition graph $\Gamma_D$). 

\textbf{Region generators:} The regions $\left\lbrace R_j \right\rbrace$ of the coarse decomposition graph $\Gamma_D$.

\textbf{[Relations]} 

\textbf{Local coarse relations:} $W_t \in \llbracket P_t ~,~ Q_t \rrbracket_{(R_l(t)^{-1}R_r(t))}$, where $\llbracket P_t ~,~ Q_t \rrbracket$ is a universal range of $t$. 

\textbf{Global cycle relations:} $W_n^{\pm1}...W_1^{\pm1}=1$, if the edge-path $W_n^{\pm1}...W_1^{\pm1}$ represents a loop in $\mathbb{R}^2$. 

\textbf{Vanishing relation:} $R_0=1$, where $R_0$ is the unbounded region generator.

Here, the edge $W^{-1}$ is the edge $W$ with the opposite orientation.

In \citep{ito}, Ito observed that when $\pi_1(\Sigma_2(S^3,\mathcal{L}(D)))$ is left-orderable, then if $R_l(t)=R_r(t)=W_t=1$, then all edge and region generators that appear in $\tilde{G}_t$ are trivial. This observation, together with the convention that the trivial group is non-left-orderable allowed Ito to give the following left-orderability criterion. 
\begin{theo}[Theorem 3.11, \cite{ito}]
Let $\mathcal{CB}$ be a coarse Brunner's presentation associated to a nice tangle-strand decomposition of a link diagram $D$. If $\pi_1(\Sigma_2(S^3,\mathcal{L}(D)))$ is left-orderable, then at least one region generator in $\mathcal{CB}$ is non-trivial.
\label{thito}
\end{theo}
\section{Main results and some applications}\label{Main_results}
In this section, we state our main theorems and then we give some applications.
\subsection{Main Theorems.}
\begin{theo}
If $T$ is a connected alternating tangle and if $(B,\tau(T))$ is its alternating encirclement, then for any slope $\alpha = p \tilde{\gamma}_{\infty} + (p+q) \tilde{\gamma}_0$ on the torus $\partial(\Sigma_2(B,\tau(T)))$ such that $\frac{p}{q} \leq 1$, the manifolds $\Sigma_2(B,\tau(T))(\alpha)$ and $\Sigma_2(B,\tau(T))(\frac{1}{\alpha})$ have non-left-orderable fundamental groups.
\label{theo1}
\end{theo}

\begin{theo}
If $T$ is a connected alternating tangle and if $(B,\tau(T))$ is its alternating encirclement, then for any slope $\alpha = p \tilde{\gamma}_{\infty} + (p+q) \tilde{\gamma}_0$ on the torus $\partial(\Sigma_2(B,\tau(T)))$ such that $\frac{p}{q} \leq 1$, the manifolds $\Sigma_2(B,\tau(T))(\alpha)$ and $\Sigma_2(B,\tau(T))(\frac{1}{\alpha})$ are $L$-spaces.
\label{theo2}
\end{theo}
In particular, by using the tangles $T_n$ shown in Fig.\ref{figadd2}, where $n$ is a positive integer, and the boxes stand for the vertical rational tangles $-\frac{1}{2n}$ or $-\frac{1}{2n-1}$, Theorems \ref{theo1} and \ref{theo2} provide a non-Seifert, non left-orderable $L$-spaces as stated in the following proposition.
\begin{figure}[H]
\centering
\includegraphics[width=0.5\linewidth]{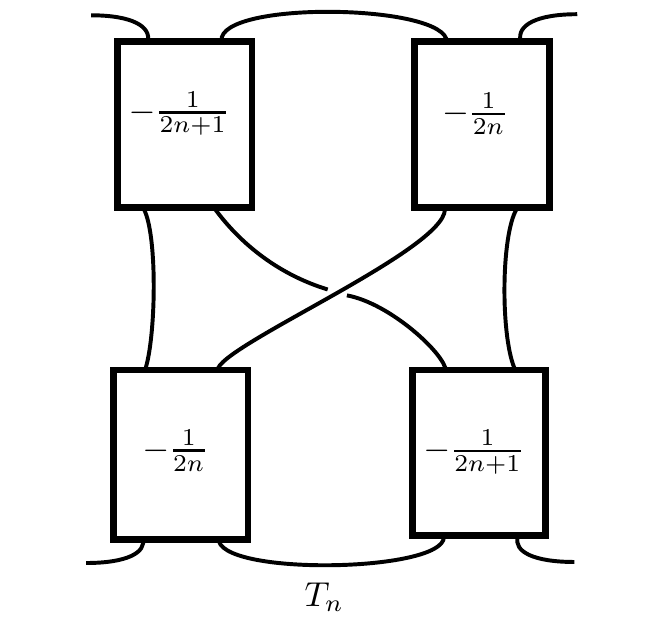} 
\caption{The tangle diagram $T_n$.}
\label{figadd2}
\end{figure}

\begin{prop}
Let $n$ be a positive integer and let $(B,\tau(T_n))$ be the alternating encirclement of the tangle $T_n$. Then for any slope $\alpha = p \tilde{\gamma}_{\infty} + (p+q) \tilde{\gamma}_0$ on the torus $\partial(\Sigma_2(B,\tau(T_n)))$ such that $\frac{p}{q} < 1$ and $p$ is even, the manifolds $\Sigma_2(B,\tau(T_n))(\alpha)$ and $\Sigma_2(B,\tau(T_n))(\frac{1}{\alpha})$ are non-Seifert fibered spaces.
\label{propadd}
\end{prop}
Before giving the proofs of these results in the next section, we look into some particular cases.
\subsection{Special cases}
Among the rational homology $3$-spheres we have considered in Theorems \ref{theo1} and \ref{theo2}, there are many Seifert fibered spaces. In the following proposition we give a surgery description of these $3$-manifolds in some particular cases.

\begin{prop}
If $T$ is the rational tangle $-\dfrac{a}{b}$ such that $a,b > 0$, then the manifold $\Sigma_2(B,\tau(T))(\frac{1}{2})$ is a Seifert fibered space. Moreover, if $T$ is an integer tangle, then there exists an integer $k$, $|k|\ge 1$, such that the manifold $\Sigma_2(B,\tau(T))(\frac{1}{2})$ can be obtained from the sphere $S^3$ by an integer surgery along the torus knot $T(2,2k+1)$. 
\label{prop1}
\end{prop}

In order to show Proposition \ref{prop1}, we will need some background and Lemma \ref{lem3}. 

Let $D$ be a reduced alternating projection of a nontrivial non-split alternating link $L$. Let $J$ be an embedded circle in the complement of $L$ in $S^3$ such that $J$ intersects the projection plane in two points and bounds a disk that lies in a plane perpendicular to the projection plane. The link $L \cup J$ is called an \textit{augmentation} of $L$. If $L$ is prime and non-isotopic to any torus link $T(2,k)$, then the link $L \cup J$ is called an \textit{augmented alternating link}. Recall that the torus link $T(2,k)$ is the link $D(\frac{1}{k})$ which is the only alternating torus link. Adams proved in \cite{adams} that augmented alternating links are hyperbolic. 

Let $T$ be a connected alternating tangle and $\tau(T)$ its alternating encirclement. If $\frac{p}{q}$ is a rational number such that $ 0 < \frac{p}{q} \leq 1$, then the link $N(-\frac{p}{p+q} + \tau(T))$ is isotopic to an augmentation of the alternating link $N((T*(-1)) + (-\frac{p}{q}))$ as explained in Fig. \ref{fig4}. The bottom left diagram in Fig. \ref{fig4} is called the \textit{augmented form} of the link $N(-\frac{p}{p+q} + \tau(T))$. 

\begin{figure}[H]
\centering
\includegraphics[width=0.7\linewidth]{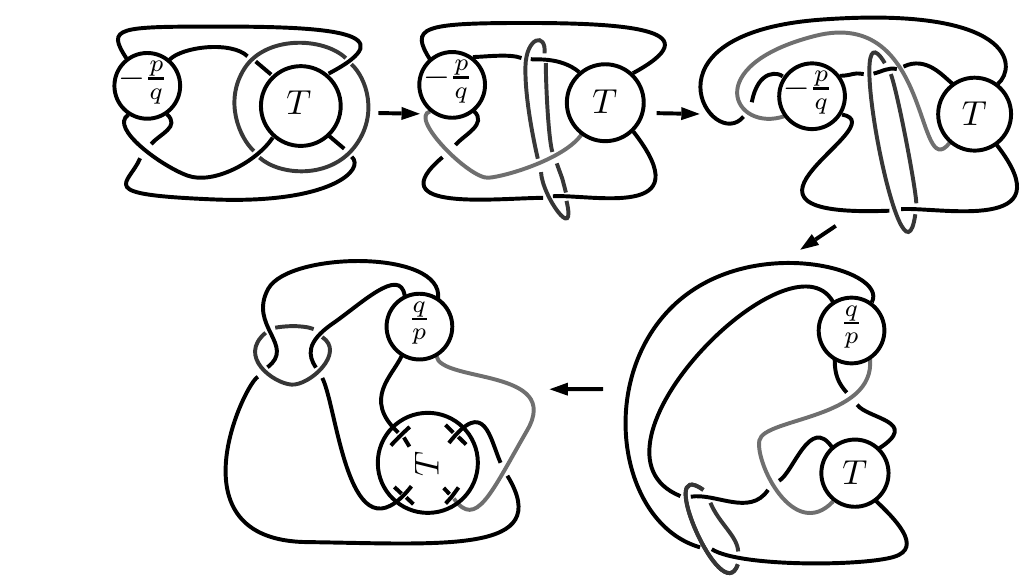}
\caption{Isotopies bringing the link $N(-\frac{p}{p+q}+\tau(T))$ into an augmentation of an alternating link.}
\label{fig4}
\end{figure}

\begin{lem}
If $T$ is a connected alternating tangle and if $\tau(T)$ is its alternating encirclement, then for any rational number $\frac{p}{q}$, $0 < \frac{p}{q} \leq 1$, the link $N(-\frac{p}{p+q} + \tau(T))$ is an augmentation of a non-trivial non-split alternating link. Moreover, if $T$ is locally unknotted, and if the determinants $N_T$ and $D_T$ satisfy $N_T > 1$ and $D_T \geq c(T)$ where $c(T)$ is the number of crossings in the tangle diagram $T$, then the link $N(-\frac{p}{p+q} + \tau(T))$ is an augmented alternating link, and so it is a hyperbolic link. 
\label{lem3}
\end{lem}
\begin{proof}
  Let $\left[ a_1,...,a_s \right]$ be the continued fraction of the rational tangle $\frac{q}{p}$. If $s=1$, then $\frac{q}{p} = m$, where $m \geq 1$ is an integer. By Proposition 3.4 in \cite{abchir}, the determinant of the alternating link $L=N((T*(-1)) + (-\frac{p}{q}))$ is equal to $mN_{T*(-1)}+D_{T*(-1)}$. It is easy to see that $N(T*(-1))$ is isotopic to $N(T)$. By Remark \ref{rmk2} we have $D_{T*(-1)} = N_T + D_T$. Finaly, we get that $\det(L)=(m+1)N_T + mD_T$. On the other hand, since $n((T*(-1)) + (-\frac{1}{m}))$ is an alternating reduced and non-split diagram of $L$, then by Corollary 1 in \cite{thistlethwaite2} the crossing number $c(L)$ is equal to the number of crossings in $n((T*(-1)) + (-\frac{1}{m}))$ which is $m+1+c(T)$, where $c(T)$ is the number of crossings in the tangle diagram $T$. Then we note that if $D_T \geq c(T)$ and $N_T > 1$, then $c(L) < \det(L)$. But, this inequality is not satisfied by the torus links $T(2,k)$ for which we have $c(T(2,k)) = \det(T(2,k)) = k$. Hence, whenever we have $D_T \geq c(T)$ and $N_T > 1$, the link $L$ will not be equivalent to $T(2,k)$ for any integer $k$. Moreover, if $T$ is locally unknotted, then the reduced alternating link diagram $n((T*(-1)) + (-\frac{1}{m}))$ is prime by Lemma 2.2 in \citep{abchir}. Consequently, if $T$ is locally unknotted, then the link $L$ is prime by Theorem 1 in \cite{menasco}.\\
Now, a simple induction argument on $s$ shows that if $N_T > 1$ and $D_T \geq c(T)$, then
$$\det(N((T*(-1)) + (-\frac{p}{q}))) > c(N((T*(-1)) + (-\frac{p}{q}))).$$
Furthermore, if $T$ is locally unknotted, then the alternating link $N((T*(-1)) + (-\frac{p}{q}))$ is prime by Lemma 2.2 in \citep{abchir} and Theorem 1 in \cite{menasco}. This shows that the link $N(-\frac{p}{p+q} + \tau(T))$ is an augmentation of a prime alternating link which is not a torus link, and hence is a hyperbolic link.
\end{proof}

\begin{proof}[Proof of Proposition \ref{prop1}]
\item Assume that $T$ is the rational tangle $-\frac{a}{b}$, $a > 0$. The augmented form of the link $N(-\frac{1}{2} + \tau(T))$, as depicted in the bottom left corner of Fig. \ref{fig4} when $\frac{p}{q}=1$, corresponds to the Montesinos link $M(0;-2,2,\frac{2a+b}{a+b})$. Then if $T$ is rational, the Dehn filling $\Sigma_2(B,\tau(T))(\frac{1}{2})$ is a Seifert fibered space. 

In Fig. \ref{fig5}, we exhibit a band-surgery on the link $D(T)$ that provides the augmented form of the link $N(-\frac{1}{2}+\tau(T))$. 

\begin{figure}[H]
\centering
\includegraphics[width=0.7\linewidth]{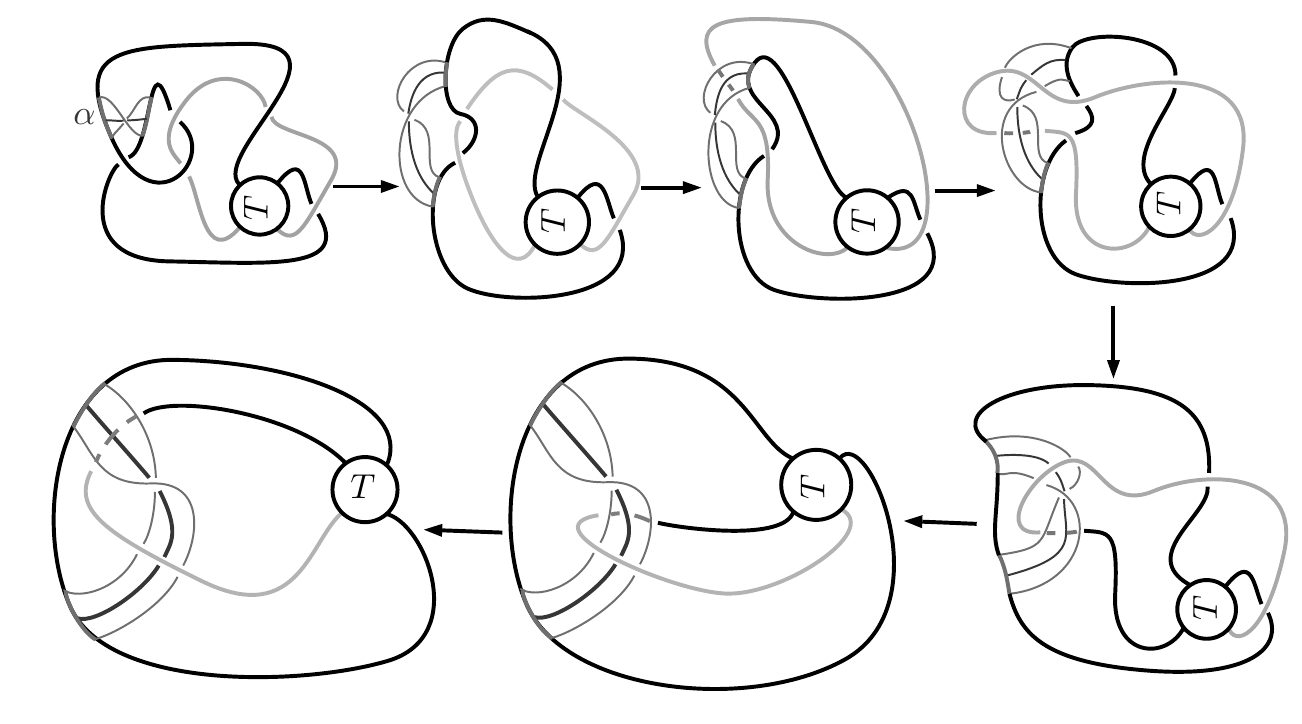}
\caption{A band-surgery on the link $D(T)$ which provides the link $N(-\frac{1}{2}+\tau(T))$.}
\label{fig5}
\end{figure}

Assume now that $\frac{a}{b}=m$. Note that in this case, the link $D(T)$ is the trivial knot $U$. Let $K_{\alpha}$ be the lift of the arc $\alpha$ depicted in Fig. \ref{fig5} in $\Sigma_2(S^3,D(T)) \cong S^3$. By the Montesinos trick, the space $\Sigma_2(B,\tau(T))(\frac{1}{2}) = \Sigma_2(S^3,N(-\frac{1}{2}+\tau(T)))$ is obtained by an integer Dehn surgery in $M$ along the knot $K_{\alpha}$. The 1-manifold $D(T)\cup\alpha$ is isotopic in $S^3$ to $U\cup\alpha'$, where $\alpha'$ is a simple arc in $S^3$ with endpoints on $U$ as depicted in the top of Fig. \ref{fig6}. The lift $K_{\alpha'}$ of the arc $\alpha'$ in $\Sigma_2(S^3,U)=S^3$ turns out to be the torus knot $T(2,-2m-1)$ as explained in the bottom of Fig. \ref{fig6}. Finally, we get that the space $\Sigma_2(B,\tau(-m))(\frac{1}{2})=\Sigma_2(S^3,N(-\frac{1}{2}+\tau(T)))$ is provided by an integer Dehn surgery on $S^3$ along the torus knot $T(2,-2m-1)$.
\begin{figure}[H]
\centering
\includegraphics[width=0.7\linewidth]{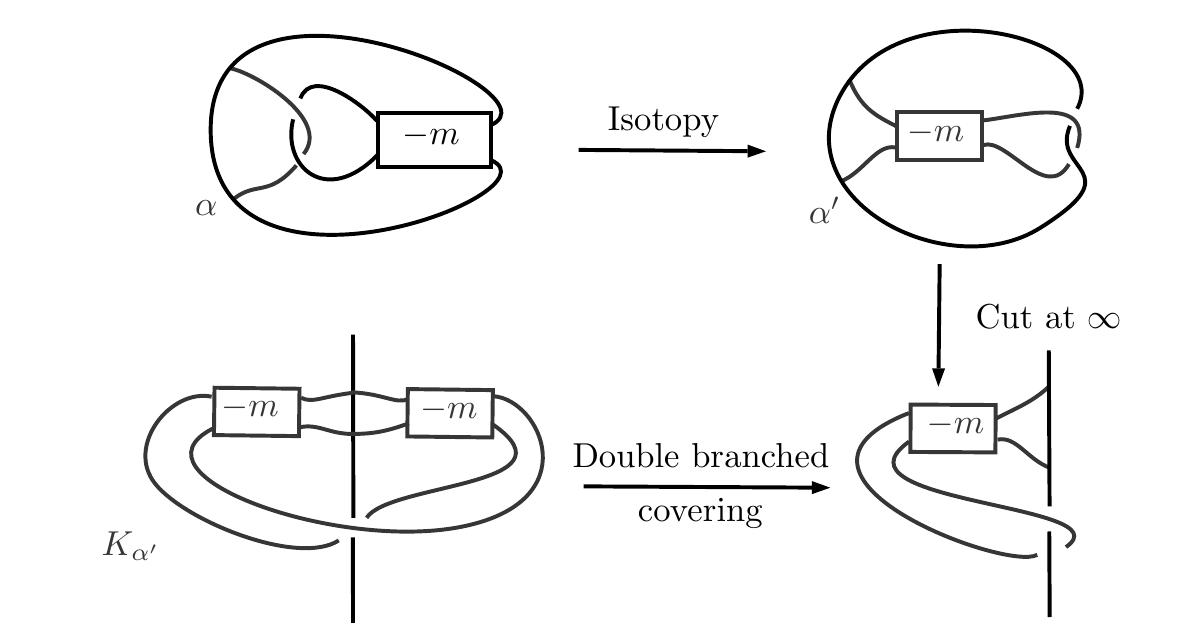}
\caption{A description of an isotopy transforming $D(T)\cup\alpha$ into $U\cup\alpha'$ (top of the figure). The arc $\alpha'$ lifts to the torus knot $T(2,-2m-1)$ (bottom of the figure).}
\label{fig6}
\end{figure}
\end{proof}

\begin{rmk}
\begin{enumerate}
\item[1.] We know by Proposition \ref{prop1} that $\Sigma_2(B,\tau(-\frac{a}{b}))(\frac{1}{2})$ is a Seifert fibered space. Moreover, Theorem \ref{theo1} shows that it is non-left-orderable and Theorem \ref{theo2} shows that it is an $L$-space. So this matches the fact that Seifert fibered spaces satisfy Conjecture \ref{conj1} (see \cite{boyer}).
\item[2.] Theorem \ref{theo2} shows that $\Sigma_2(B,\tau(-m))(\frac{1}{2})$ in an $L$-space and Proposition \ref{prop1} shows that it can be obtained by an integer Dehn surgey on $S^3$ along the torus knot $T(2,-(2m+1))$. So, we get an other way to prove that the torus knots $T(2,2m+1)$ are $L$-space knots as was firstly shown by Ozsv{\'a}th and Szab{\'o} in \cite{ozsvath}.
\end{enumerate}  
\end{rmk}

\section{Proof of main results}
At first, we note that the link $N(-\frac{p}{p+q}+\tau(T))$ has the diagram $D$ depicted in the left of Fig. \ref{fig1}. We construct a nice tangle-strand decomposition of $D$ as follows: each crossing of the tangle $T$ is regarded as a tangle part. The other tangle parts are specified in Fig. \ref{fig1}. By our sign convention, each tangle part inside $T$ is the elementary tangle $\mathcal{R}(-1)$. We denote by $\Gamma_D$ the obtained coarse-decomposition graph. 

Denote by $\Gamma_T$ the sub-graph of $\Gamma_D$ corresponding to the sub-diagram $T$. Let $E_k ... E_1$ be the uppermost edge-path in $\Gamma_T$. The coarse-decomposition graph $\Gamma_D$ is depicted in the right of Fig. \ref{fig1}. 
 
\begin{figure}[H]
\centering
\includegraphics[width=1\linewidth]{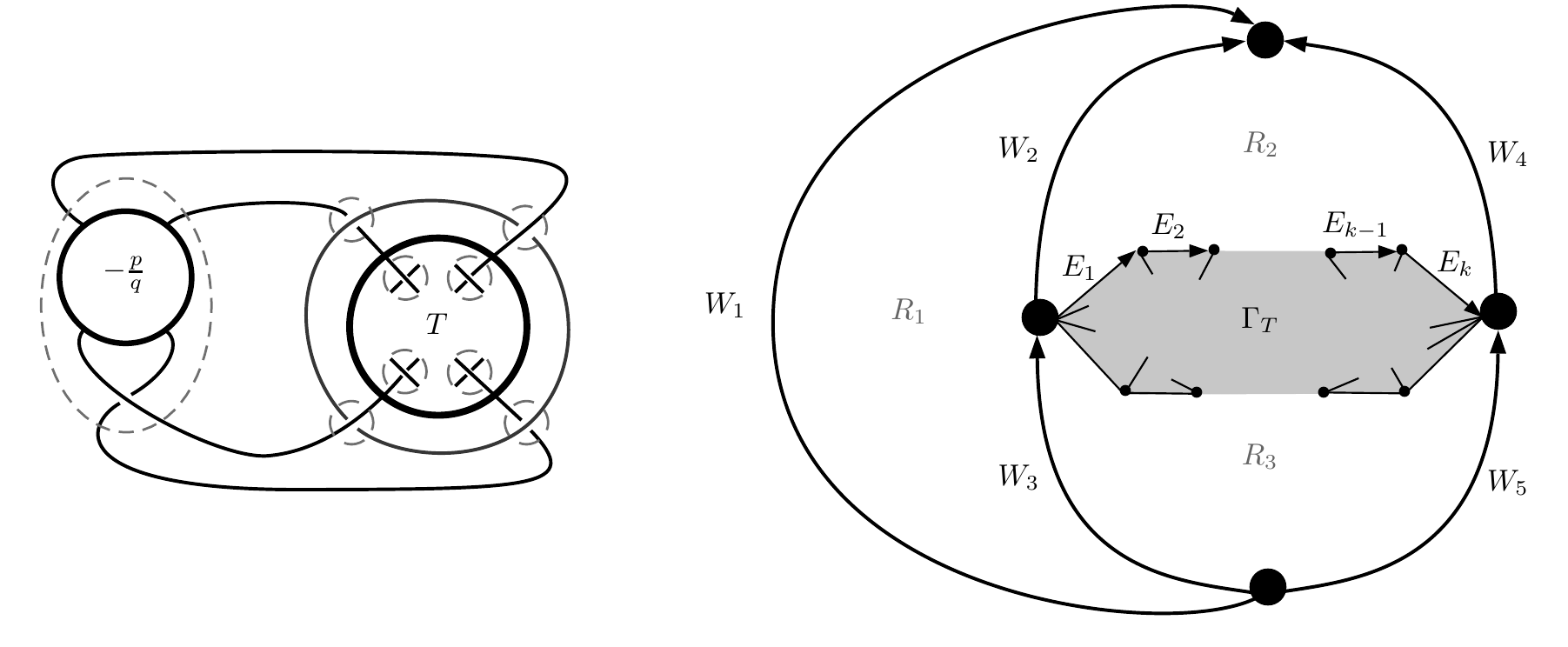}
\caption{The link diagram $D$ representing the link $N(-\frac{p}{p+q}+\tau(T))$ with a specified nice tangle-strand decomposition (left), and the associated coarse-decomposition graph $\Gamma_D$ (right).}
\label{fig1}
\end{figure}

The following is a partial description of the coarse Brunner's presentation provided by the coarse-decomposition graph on the right in Fig.\ref{fig1}.\\

\noindent\textbf{Tangle generators:} $ \left\lbrace W_i \right\rbrace_ {1 \leq i \leq 5}$.\\

\noindent\textbf{Region generators:} $\left\lbrace R_i \right\rbrace_ {1 \leq i \leq 3}$.\\

\noindent\textbf{Global cycle relations:}
The cycles in the considered graph give the relations
\begin{equation}
W_1^{-1}W_2W_3=1,\  W_5^{-1}W_4^{-1} W_2W_3=1\,\,{\rm and}\,\,W_1^{-1}W_4W_5=1.\label{cycle_relations}
\end{equation}
Hence $W_1 = W_2W_3 = W_4W_5.$\\

\noindent\textbf{Local coarse relations:} By applying Corollary 3.6 in \cite{ito} we deduce the following relations.
\begin{equation}
W_1^p = R_1^{p+q}\,,\, W_2=R_2^{-1}R_1\,,\, W_3=R_3^{-1}R_1\,,\, W_4=R_2\,,\, W_5=R_3.\label{rel_edges_regions}
\end{equation}
In particular we get that
$$W_1=W_4W_5=R_2R_3.$$
\begin{rmk}
If $t$ and $T$ are respectively a rational tangle of type 1 and an alternating tangle of type 2, the link diagram $n(t+T)$ is equivalent, up to mirror image, to the link diagram $n(\frac{1}{t} + \frac{1}{T}_c)$ as shown in the figure below. If $T$ is an alternating encircled tangle, then it is clear that $\frac{1}{T}_c$ is again an alternating encircled tangle. Hence, one can only restrict to the case where $0<\left| t \right| < 1$ when considering the numerator closure of $t$ summed with an alternating encircled tangle diagram.
\begin{figure}[H]
\centering
\includegraphics[width=1\linewidth]{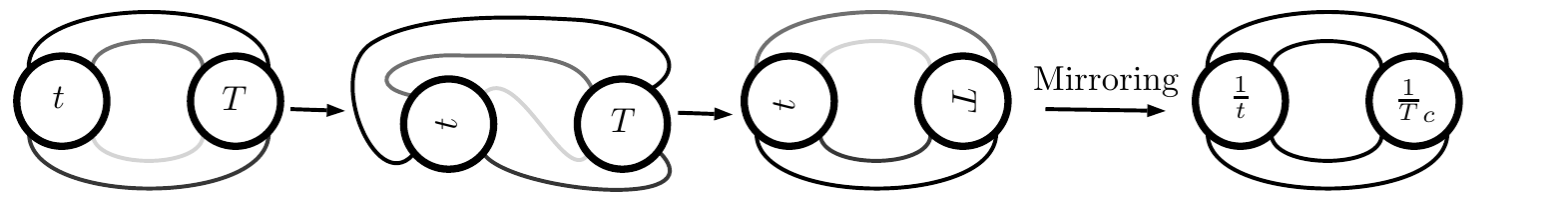}  
\end{figure}
\label{rmqbis1}
\end{rmk}

\begin{rmk}
Let $T$ be a connected alternating tangle and $(B,\tau(T))$ its alternating encirclement. Let $\alpha$ be the slope $\frac{p}{p+q}$ with respect to the standard basis for Dehn fillings of $\Sigma_2(B,\tau(T))$.
The previous remark implies that $\Sigma_2(S^3,N(\frac{p+q}{p}+\tau(T)))$ and $\Sigma_2(S^3,N(-\frac{p}{p+q}+\tau(-\frac{1}{T}_c)))$ are homeomorphic. This is equivalent to $\Sigma_2(B,\tau(T))(\frac{1}{\alpha})$ and $\Sigma_2(B,\tau(-\frac{1}{T}_c)))(-\alpha)$ are homeomorphic. This will allow us to prove our main results only for the manifold $\Sigma_2(B,\tau(T))(\alpha)$. A simple adaptation of signs in our argument will provide the same results for the manifold $\Sigma_2(B,\tau(T))(\frac{1}{\alpha})$.
\end{rmk}

The following remark allows us to reduce the cases that must be studied.
\begin{rmk} We note that if $\frac{p}{q} <0$, the manifolds $\Sigma_2(B,\tau(T))(\alpha)$ and $\Sigma_2(B,\tau(T))(\frac{1}{\alpha})$ are double branched coverings of non-split alternating links. So they are non-left-ordrerable $L$-spaces. Then we will restrict ourselves in the proofs to the case $0< \frac{p}{q}\le 1$.
\end{rmk}

\begin{lem}
If $\pi_1(\Sigma_2(B,\tau(T))(\frac{p}{p+q}))$ is left-orderable, then the region generators $R_2$ and $R_3$ have opposite signs.
\label{lem1}
\end{lem}

\begin{proof} We start from the local coarse relation $W_2=R_2^{-1}R_1$ in (\ref{rel_edges_regions}). Then $R_2^{-1}W_2=R_2^{-2}R_1$. The second global cycle relation in (\ref{cycle_relations}) yields $W_2=W_4W_5W_3^{-1}$. By using local coarse relations again, we get that  $W_2=R_2R_3R_1^{-1}R_3$. Hence $R_2^{-2}R_1= R_3R_1^{-1}R_3$. Now by using the relation $W_1=R_2R_3$ we note that:
\begin{align*}
   R_2^{-2}R_1 &= R_3R_1^{-1}R_3 \\
\Leftrightarrow  R_2^{-1}R_1R_3^{-1} &= (R_2R_3)R_1^{-1}\\
\Leftrightarrow  R_2^{-1}R_1R_3^{-1} &= W_1R_1^{-1} \\
\Leftrightarrow  R_3R_1^{-1}R_2 &= R_1W_1^{-1} \\
\end{align*}
This implies that 
\begin{equation}
(R_3R_1^{-1}R_2)^{p+q} = (R_1W_1^{-1})^{p+q}
\label{eq1}
\end{equation}
By Lemma 3.3 in \cite{ito}, the generators $W_1^{\pm1}$ and $R_1^{\pm1}$ commute and a simple induction argument shows that $$(R_3R_1^{-1}R_2)^k = R_3R_1^{-k}W_1^{k-1}R_2, \text{ for any integer } k \geq 1.$$ Finally, one can transform the equality (\ref{eq1}) and get the following:
\begin{equation}
R_3R_1^{-(p+q)}W_1^{p+q-1}R_2 = R_1^{p+q}W_1^{-(p+q)} \Leftrightarrow R_3W_1^{q-1}R_2 = W_1^{-q}
\label{eq2}
\end{equation}
We will prove the result for $R_2 \geq 1$. The other case can be shown in a similar way. 

\textbf{Case 1:} $W_1^{q-1} \leq 1$. 

Since $q \geq 1$, then necessarily $W_1 \leq 1$. Since $W_1 = R_2R_3$, this implies that $R_3 \leq R_2^{-1}$. So the result follows from the assumption that $R_2 \geq 1$. 

\textbf{Case 2:} $W_1^{q-1} \geq 1$.

In this case, one has that $R_3W_1^{q-1} \geq R_3$. On the other hand, the assumption $R_2 \geq 1$ implies that $R_3W_{1}^{q-1}R_2 \geq R_3W_{1}^{q-1}$. Hence, by the equality (\ref{eq2}) one gets that $W_1^{-q} \geq R_3$. But since, $W_1^{q-1} \geq 1$ and $q \geq 1$, then necessarily one has that $W_1 \geq 1$, which implies that $1 \geq W_1^{-q} \geq R_3$. 

This completes the proof.
\end{proof}

\begin{proof}[Proof of Theorem \ref{theo1}]
Suppose that $\pi_1(\Sigma_2(B,\tau(T))(\frac{p}{p+q}))$ is left-orderable. Assume that $R_2 \geq 1$, the other case can be shown in a similar way by interchanging the roles of $R_2$ and $R_3$. By Lemma \ref{lem1}, we have $R_3 \leq 1 \leq R_2$. 

If the graph $\Gamma_T$ has no regions (we exclude here the unbounded region of $\Gamma_T$ since we consider it only as a sub-graph), then it is clear that the only edges of $\Gamma_T$ are the edges $E_i=R_3^{-1}R_2$, $1 \leq i \leq k$. This implies that $E_i \geq 1$, for every $1 \leq i \leq k$. Consequently, we will have the following:
 
\begin{align*}
1 \leq R_2 &\leq R_2 (E_k...E_1)  \\
\Rightarrow 1 &\leq W_4(E_k...E_1) \\
\Rightarrow 1 &\leq W_2 \\
\Rightarrow R_2 &\leq R_1 
\end{align*}
We note that the third of the last inequalities is obtained by the global cycle relation $1=W_2^{-1}W_4E_k...E_1$, which comes from the boundary loop of the region $R_2$.

And also, we have
\begin{align*}
1 \leq R_3^{-1} &\leq R_3^{-1} (E_k...E_1)  \\
\Rightarrow 1 &\leq W_5^{-1}(E_k...E_1) \\
\Rightarrow 1 &\leq W_3^{-1} \\
\Rightarrow R_1 &\leq R_3 
\end{align*}
Also, in the last inequalities, we note that the third one is obtained by the global cycle relation $$1=W_3^{-1}W_5E_k...E_1,$$ coming from the boundary loop of the union of the region $R_2$ and the graph $\Gamma_T$ (shaded region in Fig. \ref{fig1}).

Finally, we get $ 1 \leq R_2 \leq R_1 \leq R_3 \leq 1$. This is a contradiction by Theorem \ref{thito}.

Assume now that the sub-graph $\Gamma_T$ has at least one region and let $R$ denote the $<$ -maximal region among the regions of $\Gamma_T$. The cycle that constitutes the boundary of the region $R$ can be expressed as $(e_l...e_1)(f_1^{-1}...f_n^{-1})=1$, where $R_r(e_i)=R_l(f_j)=R$ for every $1 \leq i \leq l$ and $1 \leq j \leq n$. We have $e_i =R^{-1}R_l(e_i)$ and $f_j=R_r(f_j)^{-1}R$. 

\textbf{Case 1:} $R \geq R_2$. In this case, we have $e_i \leq 1$ and $f_j \geq 1$ for every $1 \leq i \leq l$ and $1 \leq j \leq n$. Suppose that there exists some $1 \leq i_0 \leq l$ such that $R_l(e_{i_0}) < R$, then $e_{i_0} < 1$. Hence, we get the following:
\begin{align*}
e_{i_0}(e_{i_0 -1}...e_1) &\leq e_{i_0} < 1 \\
\Rightarrow (e_l...e_1) &< e_l...e_{i_0+1} \leq 1 \\
\Rightarrow (e_l...e_1)(f_1^{-1}...f_n^{-1}) &<1
\end{align*}

The last inequality contradicts the fact that $(e_l...e_1)(f_1^{-1}...f_n^{-1})=1$. Similarly, we get a contradiction if we suppose that there exists some $1 \leq j_0 \leq n$ such that $R_r(f_{j_0}) < R$. Finally, for all $1 \leq i \leq l$ and $1 \leq j \leq n$, we have $R_l(e_i)=R_r(f_j)=R$. This shows that any region in $\Gamma_T \cup \left\lbrace R_2,R_3 \right\rbrace$ which shares an edge with $R$ is equal to $R$. If we adapt the previous argument to the regions sharing edges with $R$, we show that each region that shares an edge with these regions is again equal to $R$. We iterate this process until we show that all regions of $\Gamma_T \cup \left\lbrace R_2,R_3 \right\rbrace$ are equal to $R$. Hence $1 \leq R_2=R=R_3 \leq 1$. And since $W_1 = R_2R_3=1$ and $W_1^{p} = R_1^{p+q} = 1$, then $R_1=1$. Hence, all region generators in the coarse-Bunner's presentation are trivial. This is a contradiction by Theorem \ref{thito}. 

\textbf{Case 2:} $R < R_2$. In this case, we have $E_i > 1$ for every $1 \leq i \leq k$. Hence 
\begin{align*}
1 \leq R_2 &< R_2 (E_k...E_1)  \\
\Rightarrow 1 &< W_4(E_k...E_1) \\
\Rightarrow 1 &< W_2 \\
\Rightarrow R_2 &< R_1 
\end{align*}

And also, we have
\begin{align*}
1 \leq R_3^{-1} &< R_3^{-1} (E_k...E_1)  \\
\Rightarrow 1 &< W_5^{-1}(E_k...E_1) \\
\Rightarrow 1 &< W_3^{-1} \\
\Rightarrow R_1 &< R_3 
\end{align*}

Finally, we get $ 1 \leq R_2 < R_1 < R_3 \leq 1$. Which is a contradiction

\end{proof}

\begin{lem}
If $T$ is a connected alternating tangle and $\tau(T)$ is its alternating encirclement, then
$$ N_{\tau(T)}=D_{\tau(T)}=4(N_T+D_T).$$
\label{lem2}   
\end{lem}

\begin{proof}
Since $D(\tau(T))$ is a non-split alternating link, then we can compute its determinant by smoothing one by one non-nugatory crossings. This is done in Fig. \ref{determinant} where the obtained links are labeled according to our notations. The determinant of the link $D(-\frac{1}{2}*T)$ is $ 2N_T+D_T $ by Proposition 3.4 in \cite{abchir}, and the determinant of the link $D(T) \# T(2,2)$ is equal to $\det(T(2,2)) \times D_T = 2D_T$. This gives that $D_{\tau(T)} = 2N_T+D_T + 2D_T = 4(N_T+D_T)$. Now, since the obtained formula for $D_{\tau(T)}$ is unaffected by the inverse operation on tangles, and since $N(\tau(T))$ and $D(-\frac{1}{\tau(T)}_c)$ are the same, then $N_{\tau(T)}$ is equal to $D_{\tau(T)}$.

\begin{figure}[H]
\centering
\includegraphics[width=0.7\linewidth]{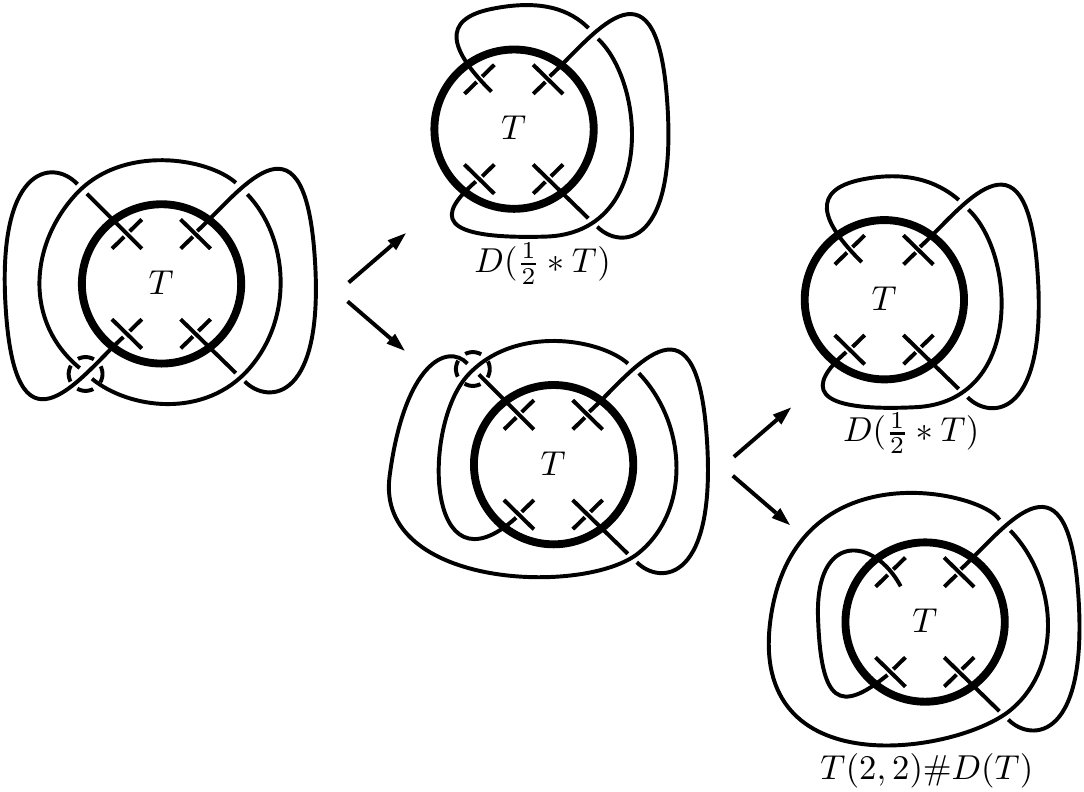}
\caption{Smoothing non-nugatory crossings of the link $D(\tau(T))$ until reaching some links with known determinants.}
\label{determinant}
\end{figure}

\end{proof}

\begin{proof}[Proof of Theorem \ref{theo2}]
Let $T$ be a connected alternating tangle and $\tau(T)$ be its alternating encirclement. Let $D$ be the link diagram $n(-\frac{1}{2}+\tau(T))$ depicted in Fig. \ref{fig2}. We will show that $D$ is quasi-alternating at the marked crossing $c$. 

\begin{figure}[H]
\centering
\includegraphics[width=0.3\linewidth]{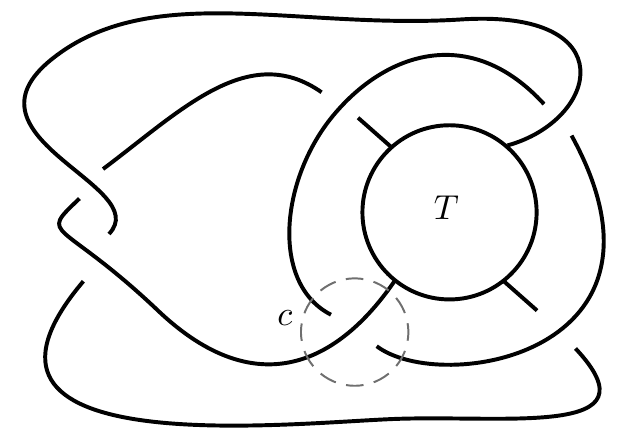}
\caption{The link diagram $D$ representing the link $N(-\frac{1}{2}+\tau(T))$.}
\label{fig2}
\end{figure}

As explained in Fig. \ref{fig3},\ref{fig3bis}, the links $\mathcal{L}(D^c_0)$ and $\mathcal{L}(D^c_{\infty})$ are respectively equivalent to the links $N(\frac{4}{3}+(-\frac{1}{T}_{cc})_v)$ and $D(T)$, which are alternating links. 

\begin{figure}[H]
\centering
\includegraphics[width=1\linewidth]{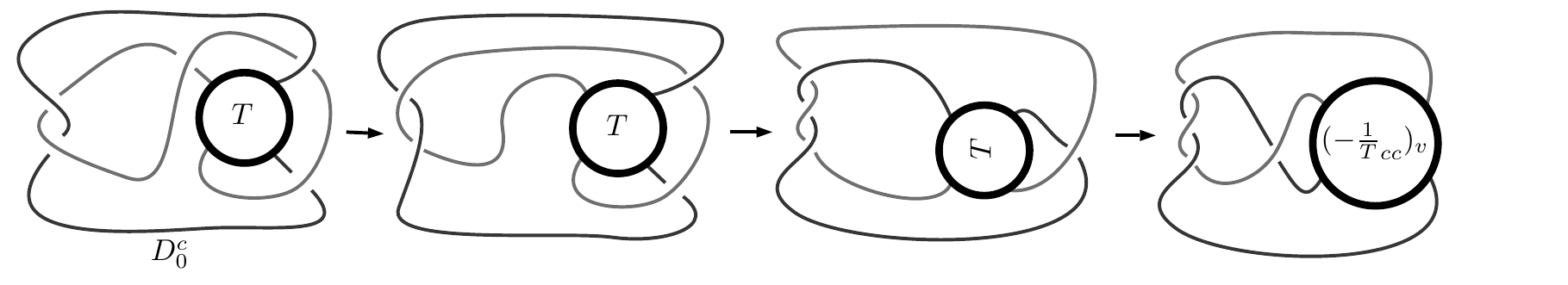}
\caption{An isotopy bringing the link $\mathcal{L}(D^c_0)$ into $N(\frac{4}{3}+(-\frac{1}{T}_{cc})_v)$.}
\label{fig3}
\end{figure}

\begin{figure}[H]
\centering
\includegraphics[width=1\linewidth]{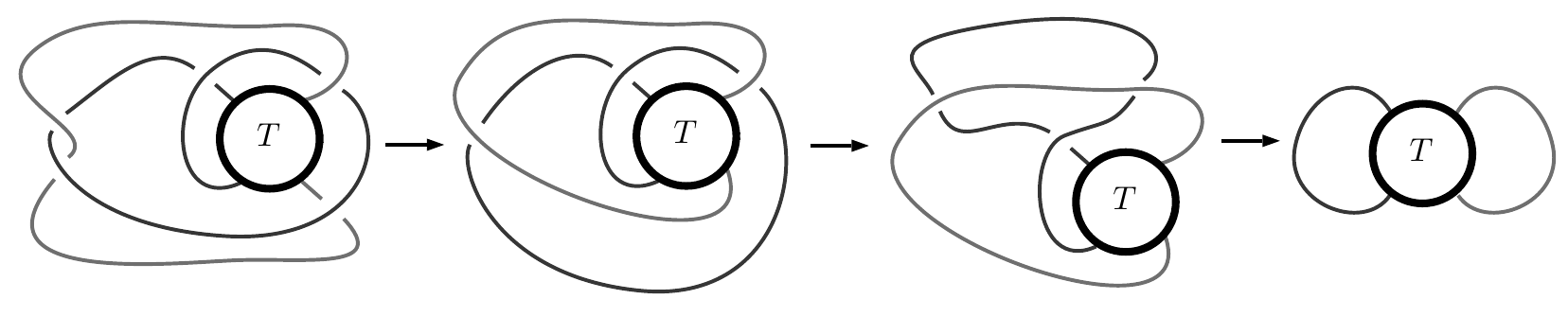}
\caption{An isotopy bringing the link $\mathcal{L}(D^c_{\infty})$ into $D(T)$.}
\label{fig3bis}
\end{figure}

By Proposition 3.4 in \cite{abchir}, we have that $\det(D^c_0)=4N_T+3D_T$, which is a non-zero integer. This shows that $\mathcal{L}(D^c_0)$ is a non-split alternating link. This is also the case for the link $\mathcal{L}(D^c_{\infty})$ by connectedness of the tangle $T$. on the other hand, by Lemma \ref{lem2}, and Proposition 3.4 in \cite{abchir}, we have that $\det(D)=\left| 2N_{\tau(T)}-D_{\tau(T)} \right| = 4(N_T+D_T)$. This shows that $\det(D)=\det(D^c_0) + \det(D^c_{\infty})$, and hence the link diagram $D$ is quasi-alternating at the crossing $c$. Now, since $\mathcal{L}(D)$ is the link $N(-\frac{1}{2}+\tau(T))$, then by Corollary 4 in \cite{abchir}, the link diagram $n(-\frac{1}{3}+\tau(T))$ is quasi-alternating at each of the three crossings of the elementary vertical tangle $-\frac{1}{3}$. We can extend the top one by the rational tangle $-\frac{p}{q-p}$ and obtain the link $N(-\frac{p}{p+q}+\tau(T))$. Thus, the link $N(-\frac{p}{p+q}+\tau(T))$ is quasi-alternating by Theorem 2.1 in \cite{champanerkar}. This shows that $\Sigma_2(B,\tau(T))(\frac{p}{p+q}) \cong \Sigma_2(S^3,N(-\frac{p}{p+q}+\tau(T)))$ is an $L$-space. 

\end{proof}

It remains for us to prove Proposition \ref{propadd}. Our main argument is based on the following remark: It is known that if the double branched covering of a link $L$ is a Seifert fibered space, then $L$ is either a Seifert link or a Montesinos link.\\
We note that, for each integer $n>1$, if $T_n$ is the tangle in Fig. \ref{figadd2}, the link $L_n=N(-\frac{p}{p+q} + \tau(T_n))$ is neither a Seifert link nor a Montesinos link. In this way, we get an infinite family of non-Seifert fibered $3$-manifolds which are non-left-orderable $L$-spaces. To do that, we will need the following lemma:
  \begin{lem}
If $ 0 < \frac{p}{q} < 1$ is a rational number such that $p$ is even, then for any integer $n > 0$, the link $L_n=N(-\frac{p}{p+q} + \tau(T_n))$ has two components one of which is trivial and the other is neither rational nor a torus knot.
\label{lemadd}
\end{lem}

Recall that a \textit{rational link} is the closure of a rational tangle. If a link diagram is the closure of a standard rational tangle diagram then it is called a \textit{standard rational link diagram}. It is shown in \cite{thistlethwaite3} (Theorem 4.1 and Proposition 5.2) that any alternating link diagram of a rational link is a standard rational diagram.

\begin{proof}
  Since the tangle diagram $T_n$ is strongly alternating, then by Corollary 1 in \cite{thistlethwaite2} we have that $c(T_n) = c(N(T_n)) = c(D(T_n))$. Hence, by Proposition 3.1 in \cite{qazaqzeh}, we have that $1 < c(T_n) \leq N_T$ and $c(T_n) \leq D_T$. Moreover, since $T_n$ is locally unknotted, then by Lemma 3.4, the link $L_n$ is an augmented alternating link. So $L_n$ has a trivial component. More precisely, it is an augmentation of the prime, alternating, and non-torus link $K_n:=N((T_n * (-1)) + (-\frac{p}{q}))$. To complete the proof of the lemma, it remains to show that $K_n$ is a non-rational knot.\\
  To see that $K_n$ is non-rational, we consider the particular diagram $D_n := n(T_n + (-\frac{p}{q}*(-1)))$ of $K_n$ shown in Fig. \ref{figadd1} (this can be easily seen by moving the crossing $(-1)$ away from $T_n$ towards the rational tangle $-\frac{p}{q}$). Moreover, since $T_n$ is strongly alternating, then by Corollary 5.1 in \cite{lickorish}, $T_n$ is non-rational. This implies that the alternating link diagram $D_n$ is not a standard rational diagram. Hence, by Theorem 4.1 and Proposition 5.2 in \cite{thistlethwaite3} the link $K_n$ is not rational for any $n >0$.\\
Now we show that $K_n$ is a knot. At first, recall that $N(T_n)$ is a knot for any $n > 0$, so the numerator closure arcs of $n(T_n)$ belong the single component of $N(T_n)$. On the other hand, since $p$ is an even integer equal to $N_{-\frac{p}{q}*(-1)}$, then by Theorem 6 and Corollary 1 in \cite{kauffman}, the link diagram $n(-\frac{p}{q} * (-1))$ has two components each containing a different numerator closure arc. Let $C_t$ (respectively $C_b$) denote the component of $N(-\frac{p}{q} * (-1))$ containing the top (respectively the bottom) numerator closure arc. One can easily see that when we join the top and the bottom endpoints of $T_n$ respectively with the top and the bottom endpoints of the rational tangle $-\frac{p}{q}*(-1)$ to build the diagram $D_n$, the two components $C_t$ and $C_b$ are inserted in the single component of $N(T_n)$ as explained in Fig. \ref{figadd1}. Then $K_n$ is a knot.

\begin{figure}[H]
\centering
\includegraphics[width=0.4\linewidth]{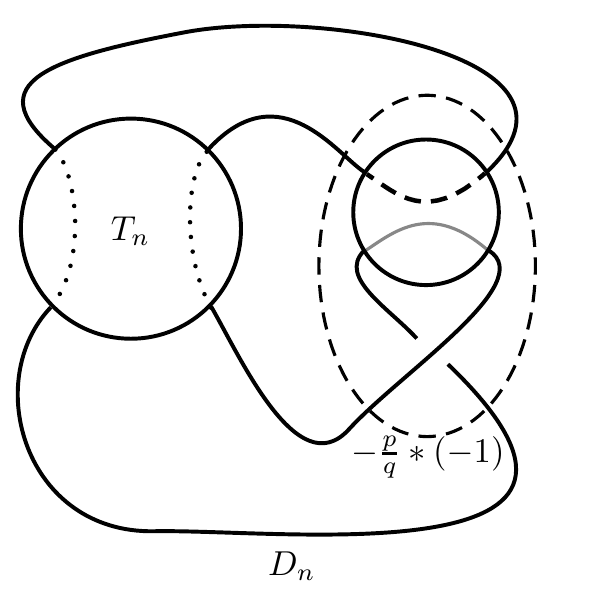} 
\caption{The knot diagram $D_n$ where the dashed (gray) arc represents a piece of $C_t$ ($C_b$).}
\label{figadd1}
\end{figure}

\end{proof}
  
  \begin{proof}[Proof of Proposition \ref{propadd}]
By Lemma \ref{lemadd}, the link $L_n$ has two components, one of which is trivial and the other is neither a torus knot nor a rational knot. Then by Lemma 2.7 in \cite{meier}, the link $L_n$ is not a Seifert link. Moreover, by Criteria 2.15 in \cite{meier}, the link $L_n$ is not a Montesinos link. Now since a link that has a Seifert fibered double branched covering is a Seifert link or a Montesinos link as mentioned in the introduction of \cite{mecchia}, then the double branched covering of the link $L_n$, which is homeomorphic to $\Sigma_2(B,\tau(T))(\alpha)$, cannot be a Seifert fibered space. The result for the manifold $\Sigma_2(B,\tau(T))(\frac{1}{\alpha})$ is deduced by Remark 5.
\end{proof}

We end this paper with two questions that are motivated by Proposition \ref{propadd}.\\

Let $T$ be a connected alternating tangle, and let $(B,\tau(T))$ be its alternating encirclement. By using the same argument as in the proof of Theorem \ref{theo1}, one can show that the filling $\Sigma_2(B,\tau(T))(\frac{k}{k+1})$, which is homeomorphic to the double branched covering of the link $L:=N(-\frac{k}{k+1} + \tau(T))$, has a non-left-orderable fundamental group for every integer $k \geq 1$. By Proposition 3.4 in \cite{abchir} and Lemma \ref{lem2}, the determinant of the link $L$ is equal to $(k+1)N_{\tau(T)} - kD_{\tau(T)} = 4(N_T + D_T)$. On the other hand, we have that $c(n(-\frac{k}{k+1} + \tau(T)))=k+1+c(\tau(T))=k+5+c(T)$. Hence, for $k > 4(N_T + D_T)-(c(T)+5)$, we will have that $c(n(-\frac{k}{k+1} + \tau(T))) > \det(L)$. By using Proposition 5.4 in \cite{abchir}, it follows that if $k > 4(N_T + D_T)-(c(T)+5)$, then the link $L$ is non-quasi-alternating. But, it may happen that the double branched covering of the link $L:=N(-\frac{k}{k+1} + \tau(T))$ is also the double branched covering of other quasi-alternating link. Consequently, the double branched covering description of the $3$-manifold $\Sigma_2(B,\tau(T))(\frac{k}{k+1})$ does not tell us wether it is an L-space or not. This fact motivates the following question.
\begin{qst}
Is the non-left-orderable $3$-manifold $\Sigma_2(B,\tau(T))(\frac{k}{k+1})$ an $L$-space for every integer $k \geq 1$?
\end{qst}
Our last discussion yields another interesting question. In 2011, Greene stated the following conjecture \cite{greene2011conway}:
\begin{conj}
  If a pair of links have homeomorphic branched double-covers, then either both are alternating or both are non-alternating.
\end{conj}
Then we can ask the similar following question for quasi-alternating links:
\begin{qst}
  Can a closed orientable $3$-manifold be the branched double-cover of both a quasi-alternating link and a non-quasi-alternating link?
\end{qst}

\nocite{*}
\bibliographystyle{plain}

\begin{center}
  Hamid Abchir\\
  Fundamental and Applied Mathematics Laboratory\\
Hassan II University. EST.\\Casablanca.\\Morocco.\\\vspace{0.25cm}
{e-mail: hamid.abchir@univh2c.ma}\\
\vspace{1.5cm}Mohammed Sabak\\
Fundamental and Applied Mathematics Laboratory\\
Hassan II University. Ain Chock Faculty of sciences.\\Casablanca.\\Morocco.\\\vspace{0.25cm}
{e-mail: mohammed.sabak-etu@etu.univh2c.ma}
\end{center}
\end{document}